\documentclass{amsart}
\usepackage{amsthm}
\usepackage{amsmath}
\usepackage{amsfonts}
\usepackage{graphicx}
\usepackage{array}
\usepackage{xcolor}

\usepackage[
backend = biber,
style=alphabetic,
maxbibnames=99,
]{biblatex}
\usepackage{hyperref}
\usepackage[all]{hypcap}

\newtheorem{theorem}{Theorem}[section]
\newtheorem{lemma}[theorem]{Lemma}
\newtheorem{definition}[theorem]{Definition}
\newtheorem{proposition}[theorem]{Proposition}
\newtheorem{corollary}[theorem]{Corollary}
\newtheorem*{repeatThmSatelliteDiagram}{Theorem \begin{NoHyper}\ref{thm:satelliteDiagram}\end{NoHyper}}
\newtheorem*{repeatThmJSJBound}{Theorem \begin{NoHyper}\ref{thm:JSJbound}\end{NoHyper}}
\newtheorem*{repeatThmBound}{Theorem \begin{NoHyper}\ref{thm:bound}\end{NoHyper}}

\title{The Crossing Number and Arc Index of Satellite Knots}
\author{Sam Ketchell}

\begin{document}

\begin{abstract}
    By considering arc presentations and grid diagrams of satellite knots, we prove that if $K$ is a satellite knot with companion $C$ and wrapping number $w$, then its crossing number satisfies $c(K) \geq \frac{2}{5}(w^2c(C))^{1/11}$. We also prove other similar bounds regarding the crossing number and arc index of satellite knots and links. This is partial progress on Problem 1.2 of the K3 Problem List.
\end{abstract}

\maketitle

\tableofcontents

\pagebreak

\section{Introduction}

An old and well-known open problem is how the crossing number of knots behaves under satellite operations. Let a knot $K$ be a satellite of a companion knot $C$ with pattern $P$, by which we mean $P$ is a knot contained in a solid torus $V$ which is not a core curve of $V$ and is not contained in any ball within $V$, and we embed $V$ into $S^3$ knotted like $C$, so that $P$ is sent to $K$. By the \emph{wrapping number} $w$ we mean the minimum number of intersections between $P$ and any meridian disc of $V$. Given diagrams of $C$ and $P$, one can easily draw a diagram of $K$. This gives an upper bound on the crossing number $c(K)$ in terms of $C$ and $P$, but lower bounds are very hard to come by. One feature of this diagram is that each crossing of our diagram for $C$ becomes a cluster of at least $w^2$ crossings in the new diagram, since we have at least $w$ strands going one way passing over at least $w$ going the other way.

For this reason one might believe that $c(K)$ must always be at least $w^2c(C)$, or less optimistically at least $c(C)$. These conjectures appear as Problem 1.2 in the K3 problem list \cite{K3ProblemList}.

The stronger conjecture has been shown to be true in some special cases. In \cite{AdequateWhitehead}, it is established for untwisted whitehead doubles of knots that have zero-writhe adequate diagrams. In \cite{AdequateCable}, it is established for cables of adequate knots. Finally, in \cite{TwistFamilies} it is established asymptotically when one considers an infinite family of satellites produced by repeatedly performing full twists.

Freedman and He proved a general bound in \cite{genusBound}, showing that for any satellite knot, $c(K)$ is at least the genus of $C$ multiplied by the square of the winding number (the number of times $P$ winds around the torus algebraically, which is at most $w$ but could be much smaller or even zero).

For general satellite knots, the best known bound in terms of $c(C)$ comes from \cite{satelliteBound} in which Lackenby proves that $c(K) \geq 10^{-13}c(C)$. Though the coefficient is quite small, this bound grows linearly in $c(C)$, which is the best growth rate we could possibly hope for.

However, one might hope for a bound that incorporates not just $c(C)$ but also the wrapping number. In this paper we prove such a lower bound on the crossing number of a general satellite knot, and a similar bound on its arc index.

\begin{theorem}
    \label{thm:bound}
    If $K$ is a satellite knot with companion knot $C$ and wrapping number $w$, then
    \[
        c(K) \geq \frac{2}{5} (w^2c(C))^{1/11}
    \]
    and
    \[
        \alpha(K) \geq \frac{4}{5} (w \alpha(C))^{1/7}.
    \]
\end{theorem}

Though these new bounds are far from linear, they have the advantage of growing as $w$ does.

One could also conjecture that $c(K)$ should depend on $c(P)$ (suitably defined) in some way. Writing down a neat formula is difficult because of interactions between the framing of the solid torus and the writhe of the diagram of $C$, but loosely speaking we might expect $c(K)$ to be ``something like $w^2c(C)+c(P)$".

In this paper we also prove inequalities of this flavour (though again with nonlinear growth rates) for a large class of satellite knots and even a wide variety of satellite links.

\begin{theorem}
    \label{thm:JSJbound}
    If $L$ is a non-split link contained in a disjoint union of solid tori $V = \bigsqcup V_i$ such that every $T_i = \partial V_i$ is JSJ in the exterior of $L$ and no $V_i$ is unknotted, and $C$, $P$, and $w$ are the corresponding companion link, pattern link, and minimal wrapping number around any $V_i$, then
    \[
        w^2c(C)+c(P) \leq \frac{1}{7}(2c(L)+2)^{10}
    \]
    and
    \[
        \max(w\alpha(C),\alpha(P)) \leq \frac{1}{2}\alpha(L)^5 + 2 \alpha(L)
    \]
    where by $c(P)$, $\alpha(P)$ we mean the minimal crossing number or arc index of a (grid) diagram of $P$ drawn in a collection of disjoint annuli (we place no restriction on the framing with which $P$ is drawn).
\end{theorem}

In the knot case, the assumption that the satellite torus is JSJ is not an unreasonable one. A satellite knot is one which has a knotted JSJ torus in its exterior, so all satellite knots are susceptible to this theorem, though perhaps not with a specific companion-pattern pair, if the knot arises from multiple distinct satellite constructions. We shall see that the only companions to which the theorem may not apply are connected sums of companions to which it does. In particular the satellite torus will always be JSJ if the companion is prime.

In the link case, the restrictions imposed by the assumptions can be more significant.

The techniques we use to prove these theorems are those of open book foliations, in which we take an open book decomposition of $S^3$ and examine the singular foliation it induces on a surface (in our case a collection of satellite tori for a link). These techniques were first used by Bennequin in \cite{Bennequin1989} to prove the existence of non-standard contact structures on $\mathbb{R}^3$, and have been used productively since, for instance by Birman and Menasco in \cite{BirmanMenasco} to study satellite braids, by Ito in \cite{Ito2011} to establish a relationship between the genus of a closed braid and the complexity of the braid word, and by Dynnikov and Prasolov in \cite{bypasses} to prove various results regarding grid diagrams, Legendrian knots, and braids. LaFountain and Menasco have written a book on the subject \cite{BraidFoliationsBook}.

Most relevant to our argument is \cite{monotonicSimplification}, in which Dynnikov uses these techniques to establish that if one has a grid diagram of an unknot, a composite link, or a split link, it can be transformed into a trivial, composite, or split grid diagram respectively by a sequence of elementary moves which at no point raises the complexity of the diagram. An overview of the argument is that these types of link are characterised by the presence of a certain kind of surface in the exterior: a spanning disc, a composition sphere, or a splitting sphere. These surfaces are put into a so-called ``admissible" form with respect to an arc presentation of the link, and the foliation is examined.

It is explained how, if certain patterns are present in the foliation, the surface can be simplified without raising the arc index of the link. Once all of these simplifying moves have been exhausted, the surface is in a ``nice" position which forces the grid diagram of the link to be of the special kind desired.

It is tempting, therefore, to try applying this method to satellite knots, which are characterised by the presence of an essential torus in their exterior. We may hope that by examining the foliation of the torus and performing the simplifications, we could change any grid diagram of a satellite knot, without increasing the arc index, into a ``satellite diagram", by which we mean that the grid diagram lives in a small neighbourhood of a grid diagram of the companion knot, and all of its corners live in small neighbourhoods of the corners of this diagram of the companion (see Figure \ref{fig:satelliteDiagram} for a schematic).

Unfortunately, the approach of \cite{monotonicSimplification} does not work out of the box. We can indeed simplify our torus into a ``nice" form, but this does not force our diagram to become a satellite diagram. In fact it is untrue that grid diagrams of satellite knots can always be simplified monotonically to a satellite diagram. In \cite{kazantsev}, Kazantsev gives an explicit example of a grid diagram (of arc index 11) of a satellite knot (a cable of the trefoil) such that any sequence of elementary moves transforming it into a satellite diagram must at some point raise the arc index.

Nevertheless, we find that by simplifying the torus even further, we can force our diagram into the required form. We know from Kazantsev's example that we will not be able to do this without increasing the arc index of the diagram, but it is possible to gain control over how large this increase might need to be.

Our methods work just as well for some satellite links as they do for satellite knots (though some assumptions are necessary), and we obtain the following, which is the underlying theorem from which Theorems \ref{thm:bound} and \ref{thm:JSJbound} are deduced.

\begin{theorem}
    \label{thm:satelliteDiagram}
    If $L$ is a non-split link of arc index $\alpha$ contained in a disjoint union of solid tori $V = \bigsqcup V_i$ such that every $T_i = \partial V_i$ is JSJ in the exterior of $L$ and no $V_i$ is unknotted, then there is an isotopy of $(L,T)$ in $S^3$ such that afterwards $L$ is in an arc presentation of arc index at most $\frac{1}{2}\alpha^5+2\alpha$ and the corresponding grid diagram is a satellite diagram of a grid diagram corresponding to a collection of core curves of each $V_i$.
\end{theorem}

Our proof proceeds as follows: We put our link into an arc presentation of arc index $\alpha$ and make $T$ admissible. In Section \ref{sec:admissible} we outline the theory of arc presentations and admissible surfaces that we will need, and in Section \ref{sec:admissibleTori} we describe the particulars of the theory in our situation. In Section \ref{sec:repeatedArcs} we show how, if $T$ is not already in ``narrow position" (our final goal) but has the property that there is a non-trivial curve on $T$ comprised of only two leaves of the foliation, then we can isotope it into narrow position directly. In Section \ref{sec:noRepeatedArcs} we see that if $T$ does not have this property, we can perform an isotopy to grant it the property, which we accomplish by finding a suitable meridian curve and isotoping $T$ to give this curve the required form. During this isotopy we may need to increase the arc index of $L$ by an amount that we control in terms of $\alpha$ and the binding weight of $T$. Then in Section \ref{sec:bindingWeight} we establish a bound on this binding weight in terms of $\alpha$, and this is the final ingredient we need to prove Theorem \ref{thm:satelliteDiagram}.

In Section \ref{sec:bound} we put lower bounds on the arc index of a satellite diagram in terms of $C$ and $P$, obtaining Theorem \ref{thm:JSJbound}. In Section \ref{sec:nonJSJ} we explain how, in the knot case, we can still obtain a bound if $T$ is not JSJ, obtaining Theorem \ref{thm:bound}.

Finally in Section \ref{sec:supplementals} we give two extra consequences of our discussion, including a somewhat novel (though not practical) algorithm for satellite knot detection.

\bigskip

\noindent\textbf{Acknowledgements.} The author would like to thank Marc Lackenby for many useful conversations, as well as Panos Papazoglou and Marc Kegel for independently pointing out that these bounds lead to an algorithm for recognising satellite knots.

\noindent\rule{2.1cm}{1pt}

{\footnotesize No LLMs were used in the creation of this article.}

\section{Arc presentations and admissible surfaces}
\label{sec:admissible}

\subsection{Arc presentations}

The 3-sphere $S^3$ can be thought of as the join of two circles $S_\phi^1$ and $S_\theta^1$. That is, the space
\[
\{(\phi,\tau,\theta) \mid \phi \in S_\phi^1, \tau \in [0,1], \theta \in S_\theta^1\} / \sim
\]\[
\text{where } (\phi,0,\theta_1) \sim (\phi,0,\theta_2) \text{ and } (\phi_1,1,\theta) \sim (\phi_2,1,\theta).
\]

The circle $S_\phi^1$ is naturally included in this space as the subspace with $\tau=0$, while $S_\theta^1$ is identified with the subspace $\tau=1$. We will call $S_\phi^1$ the \emph{binding circle}.

The space $S^3 \setminus S_\phi^1$ decomposes into open discs
\[D_t = \{(\phi,\tau,t) \in S^3 \mid \phi \in S_\phi^1, \tau > 0\}\]
whose closures have $S_\phi^1$ as their boundary. Thus we have an open book decomposition of $S^3$ where each page is a disc with a particular $\theta$ coordinate, and $S_\phi^1$ is the binding circle.

Frequently we will think of $t$ (determining which page of the open book decomposition a point is in) as being a time parameter, and will speak of things happening ``after" other things, or events being pushed into the past or future. Since $S_\theta^1$ is a circle, the flow of time is cyclic.

\begin{definition}
    An \emph{arc presentation} of a link $L$ is a link $L'$ isotopic to $L$ such that
    \begin{itemize}
        \item $L' \cap S_\phi^1$ is a finite set of points, which we call \emph{vertices}.
        \item for every $t$, $L' \cap D_t$ is either empty or is a single open arc whose ends approach two distinct vertices.
    \end{itemize}
    The \emph{arc index} of an arc presentation is the number of points on $S_\phi^1$ it meets, or equivalently the number of pages that it intersects in an arc.
    The \emph{arc index} $\alpha(L)$ of $L$ is the minimum arc index among all of its arc presentations.
\end{definition}

Arc presentations of a link correspond to \emph{grid diagrams} of that link, which are link diagrams that are comprised only of vertical and horizontal arcs, and with the properties that no two horizontal arcs share the same vertical coordinate, no two vertical arcs share the same horizontal coordinate, and wherever there is a crossing the vertical arc passes over the horizontal arc.

Vertical arcs correspond to arcs of the arc presentation, with their horizontal coordinate being the $\theta$ coordinate of the relevant page, while horizontal arcs correspond to vertices of the arc presentation, with their vertical coordinate being the $\phi$ coordinate of the vertex. For an example, see Figure \ref{fig:arcGrid}.

\begin{figure}
    \centering
    \includegraphics[width=0.7\linewidth]{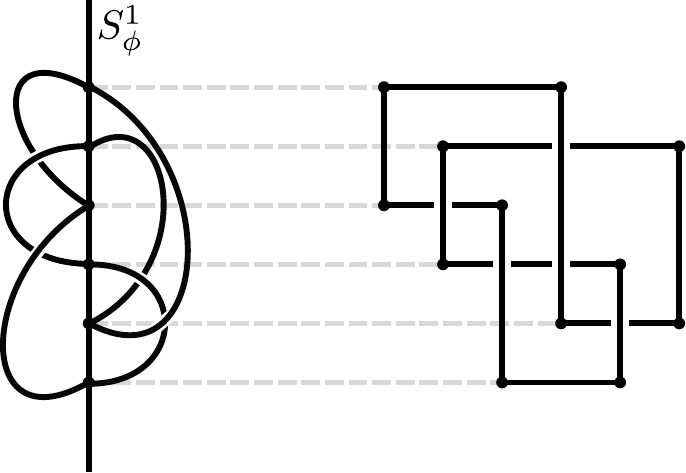}
    \caption{An arc presentation of the figure-eight knot and its corresponding grid diagram. The arc index of this arc presentation is equal to the arc index of the figure-eight knot, which is 6. Here we view the binding circle as passing through the point at infinity, so that we see it as a straight vertical line.}
    \label{fig:arcGrid}
\end{figure}

Any two grid diagrams or arc presentations of the same link are related by a sequence of elementary moves of the following types:
\begin{itemize}
    \item Stabilisation: Replacing an arc of a grid diagram with three arcs, the second of which is very short (in the sense that no other arcs have coordinates in the interval it crosses). In the case where the new short arc is vertical, this corresponds to splitting a vertex into two vertices that are adjacent on the binding circle and which are connected by a short arc in some page. In the case where the new short arc is horizontal, this corresponds to having an arc of the arc presentation make a detour to some point of the binding circle, at which it moves to a nearby page and then continues to the original destination.
    \item Destabilisation: The inverse operation of stabilisation.
    \item Column exchange: Swapping the horizontal positions of two adjacent vertical arcs, as long as these arcs do not interleave (we say the arcs do not interleave if the intervals of vertical coordinates they span are either disjoint or nested). This corresponds to pushing an arc through time past another arc, as long as doing so won't cause them to collide during the isotopy.
    \item Row exchange: Swapping the vertical positions of two adjacent horizontal arcs, as long as these arcs do not interleave. This corresponds to moving one vertex past another on the binding circle, as long as this can be accomplished without a collision.
    \item Cyclic permutation: Moving an arc from one edge of the diagram to the other edge. This corresponds to leaving an arc presentation unchanged, but choosing a different arbitrary ``starting point" for $\phi$ or $\theta$.
\end{itemize}

\subsection{Admissible surfaces}

We will want to put our collection of satellite tori into \emph{admissible form}, as defined in \cite{monotonicSimplification}. Dynnikov's definition has several constraints that are automatically satisfied if the surface in question is closed, so we omit them.

\begin{definition}
    A closed surface $S$ in $S^3 \setminus L$, where $L$ is an arc presentation, is \emph{admissible} if
    \begin{itemize}
        \item $S$ is smooth everywhere.
        \item $S$ intersects the binding circle transversely at finitely many points, which we call \emph{vertices}.
        \item the singular foliation $\mathcal{F}$ on $S \setminus S_\phi^1$ defined by $d\theta=0$ has only finitely many singularities, which are points of tangency of $S$ with pages $D_t$.
        \item all singularities of $\mathcal{F}$ are of Morse type (that is, local extrema or saddles of the circle-valued function $\theta$).
        \item each page $D_t$ contains at most one of: an arc of $L$, or a singularity of $\mathcal{F}$.
    \end{itemize}
\end{definition}

The appearance of the foliation $\mathcal{F}$ surrounding a vertex, saddle, or pole is shown in Figure \ref{fig:singularities}.

\begin{figure}
    \centering
    \includegraphics[width=0.7\linewidth]{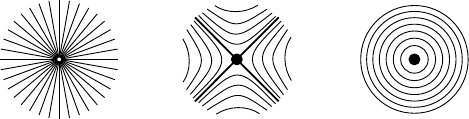}
    \caption{On the left is the region of $\mathcal{F}$ surrounding a vertex, where the binding circle meets $S$ transversely. In the centre is the region of $\mathcal{F}$ surrounding a saddle. On the right is the region of $\mathcal{F}$ surrounding a pole.}
    \label{fig:singularities}
\end{figure}

Being admissible is a generic condition for smooth surfaces in $S^3 \setminus L$, so we may assume our collection of satellite tori has these properties by performing a small perturbation if it does not.

The \emph{binding weight} of an admissible surface is its number of vertices.

We say a leaf of $\mathcal{F}$ is a \emph{separatrix} if it is incident to a saddle at one end. Due to the condition that no two saddles may lie in the same page, it is not possible for a separatrix to be incident at both ends to distinct saddles. Thus every separatrix either joins a saddle to a vertex, or joins a saddle to itself.

The \emph{valence} of a vertex is the number of separatrices that are incident to it. Each saddle has an orientation associated with it according to whether the orientation of $S$ agrees or disagrees with the orientation of $D_t$ at that point. In our case, where each satellite torus bounds a solid torus in $S^3$ on one side, and a knot exterior on the other, the orientation of a saddle is determined by whether the solid torus lies immediately to the future or to the past of the saddle.

We note in the following lemma that in our setting, we may arrange by isotoping $S$ that $\mathcal{F}$ has no poles, that no separatrix joins a saddle to itself and no leaves are circles (this is also true for other types of surface, such as those Dynnikov is concerned with in \cite{monotonicSimplification}, but this is not relevant for our purposes).

\begin{lemma}
\label{lem:noPolesOrClosedLeaves}
    If $T$ is an admissible collection of satellite tori for a non-split arc presentation $L$, then $T$ can be isotoped in $S^3 \setminus L$ so that $\mathcal{F}$ has no poles, separatrices joining a saddle to itself, or closed (circular) leaves.
\end{lemma}
\begin{proof}
    First we observe how poles may occur in $\mathcal{F}$.
    
    Consider the component of ($S \setminus \text{separatrices}$) containing a pole. This will be the pole together with some circular leaves of $\mathcal{F}$, and its boundary will either be a separatrix that connects a saddle to itself, or a pair of separatrices that connect a saddle to itself, as depicted on the left and right of Figure \ref{fig:smoothPole} respectively. In the first case, we can smooth the pole out and cancel it with the saddle, changing the foliation locally as the centre of Figure \ref{fig:smoothPole} shows.

    \begin{figure}
        \centering
        \includegraphics[width=0.8\linewidth]{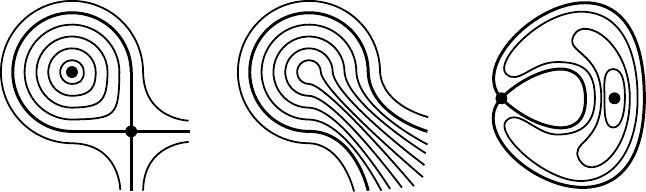}
        \caption{On the left is the region of the foliation surrounding one type of pole. After an isotopy to ``smooth out" the pole, it cancels with the saddle, and the foliation is as pictured in the centre. On the right is the region surrounding the other type of pole.}
        \label{fig:smoothPole}
    \end{figure}
    
    The second type of pole is not smoothable in a similar way. However, we note that if $\mathcal{F}$ has no separatrix connecting a saddle to itself, it also will not have any poles.

    Say $\mathcal{F}$ has a separatrix connecting a saddle to itself. This leaf is contained in a page $D_t$. This page could possibly contain two such separatrices, and they may be nested. In this case we take the innermost one. This leaf bounds a disc $D$ in the page $D_t$ that does not meet $S$ (because the leaf was innermost) or $L$ (since no arc of $L$ lies in the same page as a saddle of $\mathcal{F}$). Since $T$ is incompressible in $S^3 \setminus L$, this leaf also bounds a disc $D' \subseteq T$. Then $D \cup D'$ can be perturbed to be a smooth sphere disjoint from $T$ in $S^3$ and so $D \cup D'$ bounds a ball whose interior is disjoint from $T$ and $L$ (since $L$ is non-split). Thus we may isotope $D'$ across this ball, replacing it with $D$. Then we can push the interior of $D$ slightly into the past or future (one of these will introduce a circle of tangency to $D_t$ along $\partial D$, but the other will not, and we choose that one), forming a pole of the first type, which we can cancel with the saddle as described above.

    This process reduces the number of saddles of $\mathcal{F}$, and so by doing such moves until no more are available, we obtain a foliation in which no separatrix connects a saddle to itself. This foliation therefore also has no poles.

    Finally, if $\mathcal{F}$ has a closed leaf, we take an innermost closed leaf among those in its page. This bounds a disc in $D_t$ but also in $T$ and just as above we may isotope the disc in $T$ so that it lies in $D_t$. Then we may push it slightly into the future or past, introducing a pole without introducing any saddles. But this is not possible, as no poles can exist without leaves that connect a saddle to itself.

    Therefore $\mathcal{F}$ has no poles, separatrices connecting saddles to themselves, or closed leaves.
\end{proof}

Thus we may assume that poles and closed leaves are absent, and the four separatrices emerging from a saddle go to four distinct vertices. As a result, the separatrices form a bipartite graph; every saddle is connected to exactly four distinct vertices, and every vertex is connected to some number (its valence) of distinct saddles. Every complementary region of the separatrices is a square with boundary consisting of two vertices and two saddles connected by four separatrices, and we call such a region a \emph{tile}.

It should be noted that the word ``tile" has been used by different authors to describe either a square region with four vertices as corners and a saddle in its centre or one with two vertices and two saddles as corners. We use the latter meaning throughout.

We say two vertices are adjacent in $\mathcal{F}$ if there are leaves leading from one to the other, or equivalently if they both lie in the boundary of some common tile. Similarly we say two saddles are adjacent if they both lie in the boundary of some common tile.

The \emph{star} of a vertex is the union of all the leaves emerging from it. The saddles contained in the star of a vertex are exactly those that it is connected to by a separatrix.

\subsection{Arcs and events}

One way to think of an admissible surface of this kind (and to determine an isotopy class of such surfaces by finite data) is to consider the intersection of $S$ with each page $D_t$. Assuming as in our case that $S$ has no closed leaves, for pages that do not contain a saddle of the foliation, $S \cap D_t$ consists of a collection of arcs joining all of the vertices on the binding circle together in pairs (note that this means the binding weight is guaranteed to be even).

We can think of an admissible surface as such a collection of arcs that changes with time. If two pages are not separated by a page containing a saddle, then the arcs in those pages will be the same (up to isotopy). However, when we pass a page containing a saddle, the effect on the collection of arcs is for two of the arcs to ``merge together, then split apart" to form a different two arcs. For instance, if $v_1,v_2,v_3,v_4$ are ordered counter-clockwise around the binding circle, then when we pass a page containing a saddle whose four separatrices lead to $v_1,v_2,v_3,v_4$, the change to the collection of arcs will be to replace the arcs $v_1v_2$ and $v_3v_4$ with $v_1v_4$ and $v_2v_3$, or possibly the latter two arcs will be replaced with the former two.

In this way a saddle corresponds to an \emph{event} that happens to the collection of arcs.

As we move around $S_\theta^1$, we pass each saddle of the foliation, and the corresponding events happen in sequence, until eventually we return to our starting point and the collection of arcs has returned to its original state.

Conversely, if we have some sequence of collections of arcs, where each differs from the last by one of these saddle-style events in which two arcs merge and split to form two others, and such that the final collection in the sequence is equal to the first, then this can be realised as an admissible surface.

Thus one can draw an admissible surface by drawing the state of the collection of arcs between each event, and a few diagrams of this type will be seen in Section \ref{sec:admissibleTori} (see Figures \ref{fig:thinFig8}, \ref{fig:kazantsev}).

In this framework, the valence of a vertex $v$ is equal to the number of events in which the arc emerging from $v$ is involved, which can also be thought of as the number of times the arc emerging from $v$ changes its destination.

Frequently it is useful to regard pages containing arcs of $L$ as also being events, in which the arc of $L$ appears for an instant and then disappears again immediately.

\subsection{Simplifying moves}
\label{subsec:simplifyingMoves}

Lemma 6 in \cite{monotonicSimplification} describes how an admissible surface may be simplified in various situations. In particular for $S$ a closed admissible surface with foliation $\mathcal{F}$ in the complement of an arc presentation $L$ we have the following:

\begin{lemma}
\label{lem:valence<4}
    If $\mathcal{F}$ has a vertex of valence less than four, and $S$ has no sphere components, there is an isotopy of $(L,S)$ in $S^3$ taking $L$ to a new arc presentation and $S$ to a new admissible surface that reduces the binding weight of $S$ by two, without raising the arc index of $L$.
\end{lemma}

The idea of the proof is that vertices of valence one do not arise, and vertices of valence zero do not arise unless $S$ has a sphere component (since the star of such a vertex is a sphere). Given a vertex of valence two, Dynnikov describes how to isotope $S$ to eliminate two vertices. In the case of a vertex of valence three, due to the odd valence, there must be a pair of adjacent saddles in the star that have the same orientation. Using this, it is possible to isotope $S$ to reduce the valence of this vertex by one (so then it has valence two and the previously mentioned isotopy can be performed). The same proof gives the following fact which will be useful.

\begin{lemma}
    If the star of a vertex contains two adjacent saddles with the same orientation, there is an isotopy of $(L,S)$ in $S^3$ that reduces the valence of this vertex by one, without raising the arc index of $L$.
\end{lemma}

\begin{corollary}
\label{cor:valence4AdjacentSaddles}
    If $\mathcal{F}$ contains a vertex of valence four such that two adjacent saddles in its star have the same orientation, then there is an isotopy of $(L,S)$ in $S^3$ reducing the binding weight by two, without raising the arc index of $L$.
\end{corollary}

The isotopies described in Lemma \ref{lem:valence<4} and Corollary \ref{cor:valence4AdjacentSaddles} strictly reduce the binding weight, and therefore we may keep applying these moves whenever possible, and the process must terminate eventually. Once this is done, we have isotoped $(L,S)$ so that $S$ has no vertices of valence less than four, and every vertex of valence four is surrounded in $\mathcal{F}$ by saddles of alternating orientations.

\subsection{The effect on the foliation}
\label{subsec:foliationEffect}

It will be useful in section \ref{sec:noRepeatedArcs} to know how the foliation is altered when we apply these simplifying moves. Precise details of what the moves involve and why they are possible can be found in Lemma 6 of \cite{monotonicSimplification}; we focus here only on what happens to $\mathcal{F}$.

When we have a vertex $v$ of valence two, we first perform an isotopy which does not alter the foliation at all, but permutes the points of $T$ and of $L$ around the binding circle so that $v$ and one of its two neighbours in the foliation are adjacent on $S^1_\phi$, not even separated by any points of $L$. Then we perform an isotopy supported in a regular neighbourhood of a leaf connecting $v$ to this neighbour, pushing the surface across $S^1_\phi$ and reducing its binding weight by two.

This modifies $\mathcal{F}$ in a regular neighbourhood of this leaf. The two vertices disappear and two poles are introduced, which we then cancel with the two nearby saddles. The overall effect can be thought of as gluing $v$ and its two neighbours together into a single vertex that inherits adjacencies to other vertices from the two neighbours. This is depicted at the top of Figure \ref{fig:foliationEffect}.

\begin{figure}
    \centering
    \includegraphics[width=0.5\linewidth]{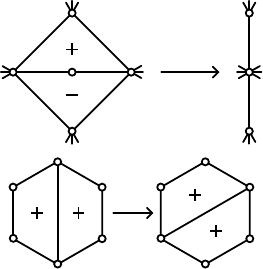}
    \caption{These are the changes that the foliation undergoes when we simplify our surface. Vertices are shown, and are connected by an edge if they are adjacent in $\mathcal{F}$. Each region has four vertices, and the sign of the saddle is indicated. The same changes can occur with saddle signs reversed.}
    \label{fig:foliationEffect}
\end{figure}

When we have two adjacent saddles of the same orientation, the squares they occupy in $\mathcal{F}$ together make a hexagon. We first push certain events through time, so that they are not in the way of a particular isotopy. The effect of this isotopy on $\mathcal{F}$ is to remove the edge between the two saddles and replace it with a different diagonal of the hexagon.

The hexagon is now split into two squares in a different way, though the signs of the two saddles remain the same. In particular, the two vertices that belonged to both squares originally have had their valence reduced. This change is depicted at the bottom of Figure \ref{fig:foliationEffect}. If one of these vertices had valence three, it now has valence two, and we can proceed to glue it and its neighbours together as described above.

\section{Admissible satellite tori}
\label{sec:admissibleTori}

In Section \ref{sec:admissible} we saw how, if certain patterns are present in the foliation $\mathcal{F}$ of an admissible surface, then it can be simplified, and as a result we can assume by isotoping that no such patterns are present.

In \cite{monotonicSimplification}, this is enough; the absence of these patterns in $\mathcal{F}$ forces the admissible surface to be in a desirable ``perfect" form which means that the associated grid diagram of the knot/link is of the kind desired.

For admissible satellite tori, this is not the case. There are plentiful examples of admissible tori which have no poles, no separatrices from a saddle to itself, no vertices of valence less than four, no vertices of valence four with adjacent coherently oriented saddles in their star, and yet are not in a ``perfect" position which forces the knot or link's grid diagram to be a satellite diagram. Therefore more methods are needed to simplify such tori.

First we will observe that while our simplified collection of tori is not necessarily perfect, it is nice in many ways (we say that it has \emph{standard foliation}). Next we will describe the ``perfect" position, which we call \emph{narrow position}, that we would like to eventually simplify the tori into. Thirdly we note an ``easy" simplification that can be made in some situations, and which we can therefore assume is unavailable.

Finally in this section we look at the example in \cite{kazantsev}, which is an explicit example of a satellite torus which is not in narrow position but is not simplifiable by the methods described in Section \ref{subsec:simplifyingMoves}, nor those described in this section.

\subsection{Standard Foliation}
\label{subsec:standardFoliation}

The simplifications from Section \ref{subsec:simplifyingMoves} may not always put our collection of satellite tori $T$ into the perfect form we need (which will be described in the next subsection), but it does let us make useful assumptions about our tori.

Once all simplifying moves have been applied, $\mathcal{F}$ has no poles, every separatrix leads from a saddle to a vertex, every tile is a square with four separatrices as its boundary, and there are no vertices of valence less than four.

Then $\mathcal{F}$ gives us a decomposition of $T$ into square tiles. Let the number of vertices, saddles, separatrices, and square tiles be $V$, $S$, $E$, and $F$ respectively, and let $V_k$ be the number of vertices of valence $k$. Since each separatrix is incident to exactly two squares, and each square is incident to four separatrices, $E = 2F$. By counting degrees at vertices and saddles, we find that $E = 2S + \sum_{k=1}^\infty \frac{k}{2} V_k$. Finally we have that $(V+S)-E+F=\chi(T)=0$.

Substituting $F=E/2$, we obtain $V+S-E/2=0$, and using our expression for $E$, we find 
\[
0 = V+S-(S+\sum_{k=1}^\infty \frac{k}{4} V_k) = \sum_{k=1}^\infty V_k - \sum_{k=1}^\infty \frac{k}{4} V_k = \sum_{k=1}^\infty (1-\frac{k}{4})V_k
\]
and therefore
\[
\sum_{k=1}^\infty (4-k)V_k = 0.
\]

As a consequence, if $T$ has no vertices of valence less than four, it also will not have any vertices of valence greater than four. Hence we can assume that our foliation $\mathcal{F}$ is a 4-regular bipartite graph between the vertices and the saddles. As a consequence, the number of vertices is equal to the number of saddles. Every vertex is connected to four saddles by four separatrices, and is therefore surrounded by four tiles, giving four adjacent vertices (although it is possible that these four may not all be distinct, as we shall see).

From the perspective of an evolving collection of arcs, this means that there are exactly as many events as vertices. Each event involves four vertices, and each vertex is involved in four events, so the other vertex to which it is connected changes four times as we go around $S_\theta^1$.

Due to Corollary \ref{cor:valence4AdjacentSaddles}, we may also assume that no two adjacent coherently oriented saddles will be in the star of any four-valent vertex. Since any two adjacent saddles are in the star of some vertex, and every vertex has valence four, this means that there is no pair of adjacent coherently oriented saddles anywhere. Regions of the foliation look like a square grid, with saddles' orientations varying in a checkerboard pattern. We say that a collection of tori in such a position has \emph{standard foliation}, and such a foliation is depicted in Figure \ref{fig:standardFoliation}.

\begin{figure}
    \centering
    \includegraphics[width=0.5\linewidth]{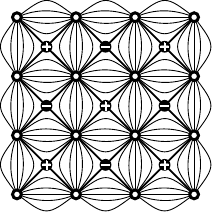}
    \caption{A small patch of a standard foliation. The empty circles are vertices, and the other circles are saddles, with a $+$ or a $-$ indicating their orientation.}
    \label{fig:standardFoliation}
\end{figure}

Our collection of satellite tori $T$ separates $S^3$ into several components. One is a link exterior of the companion link, while the others are solid tori $V_i$ which together contain $L$. When we draw diagrams of the intersection $T \cap D_t$, the arcs partition the disc $D_t$ into regions, and we shade the regions if they are in $V = \bigsqcup V_i$, and leave them unshaded if they are in $S^3 \setminus V$ (as in the upcoming Figures \ref{fig:thinFig8} and \ref{fig:kazantsev}).

We call a saddle's orientation positive if $V$ lies immediately to the future of the saddle point, and negative if $V$ lies immediately to its past. An event corresponding to a positive saddle involves two shaded regions merging to form a larger one, while an event corresponding to a negative saddle involves one shaded region splitting into two smaller ones.

The fact that the saddles around a given vertex alternate in orientation means that the four events involving a particular vertex alternate (as time progresses) between merging and splitting of shaded regions. At times this will be useful for our arguments; if a vertex has just undergone a negative event, for instance, the next time the arc $a$ from this vertex changes will be a positive event, and so we know the direction in which this arc must interact (that is, the next arc $a$ will ``merge and split" with must be on the side of $a$ on which $a$ bounds an unshaded region, so that the event will result in two shaded regions joining together).

Of particular interest will be arcs that join two vertices that are adjacent on $S_\phi^1$. We call these \emph{semicircles}, since they separate a small semicircular region from the rest of the page, and this region contains no other arcs or vertices. We may further specify that an arc is a \emph{shaded} or \emph{unshaded} semicircle depending on whether this semicircular region is contained in $V$ or not.

If $T$ has binding weight $2n$ (where $n \geq 2$), then every page contains $n$ arcs, at least two of which must be semicircles (since there must be some ``outermost" arcs on either side of any given arc). The $2n$ vertices partition the binding circle into $2n$ \emph{segments}, of which $n$ lie in $V$, while the other $n$ lie outside it. We say a segment \emph{has a semicircle} if there is some time $t$ at which the endpoints of this segment are joined by an arc (which will be a semicircle). We call a segment shaded or unshaded depending on whether it lies in $V$ or not.

The link $L$ is contained in $V$, so every point of $L \cap S_\phi^1$ lies in some shaded segment. Each arc of $L$ can be treated as an event where an arc between two points of $L \cap S_\phi^1$ appears and immediately disappears at a time $t$. The fact that $L \subseteq V$ means that such an arc will be contained in one of the shaded regions of the page $D_t$. We can draw these arcs in our diagrams if we wish, such as the red arcs in Figure \ref{fig:kazantsev}.

Sometimes it will be convenient to consider a graph in the foliation $\mathcal{F}$ whose nodes are the vertices, and where each square tile gives a corresponding edge connecting the two vertices in its boundary. Note that this graph will be bipartite, between vertices with $V$ immediately clockwise of them on $S_\phi^1$, and vertices with $V$ immediately counter-clockwise of them on $S_\phi^1$. This graph may not be simple, and in fact as we will see, when $T$ is in the ideal position many pairs of vertices are connected by a pair of edges.

\subsection{Narrow Position}
\label{subsec:narrowPosition}

The collections of tori we are concerned with are those that satisfy all the conditions outlined in Section \ref{subsec:standardFoliation}, since those are the collections of tori we do not already know how to simplify.

Our goal is to have $T$ be in \emph{narrow position}, by which we mean that it is the boundary of a small regular neighbourhood of an arc presentation of the companion link. An example is given in Figure \ref{fig:thinFig8}, and a view of how such a torus looks in $S^3$ is given in Figure \ref{fig:3DNarrowPosition}.

\begin{figure}
    \centering
    \includegraphics[width=0.7\linewidth]{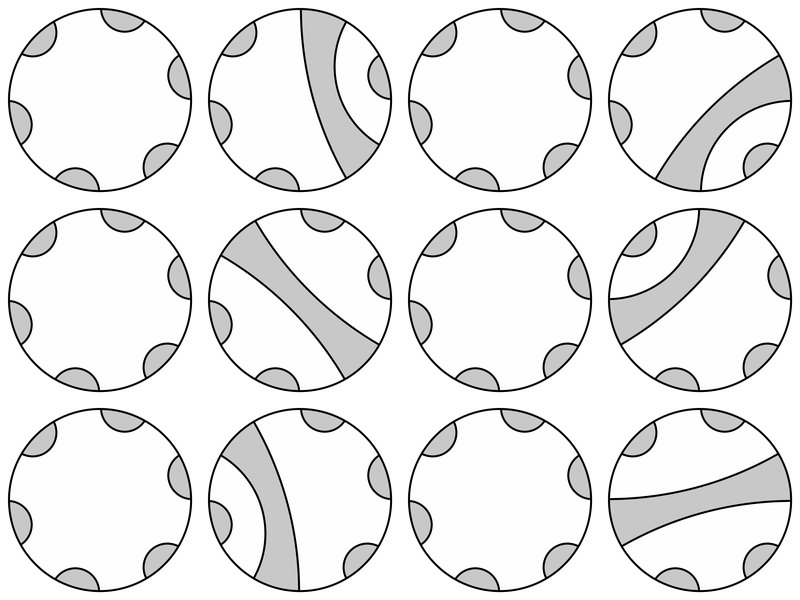}
    \caption{A diagram of a torus in narrow position. The grid should be read from left to right, then top to bottom. Each layout differs from the previous by one saddle-event. This torus is following an arc presentation of a figure-eight knot, of arc index 6.}
    \label{fig:thinFig8}
\end{figure}

\begin{figure}
    \centering
    \includegraphics[width=0.5\linewidth]{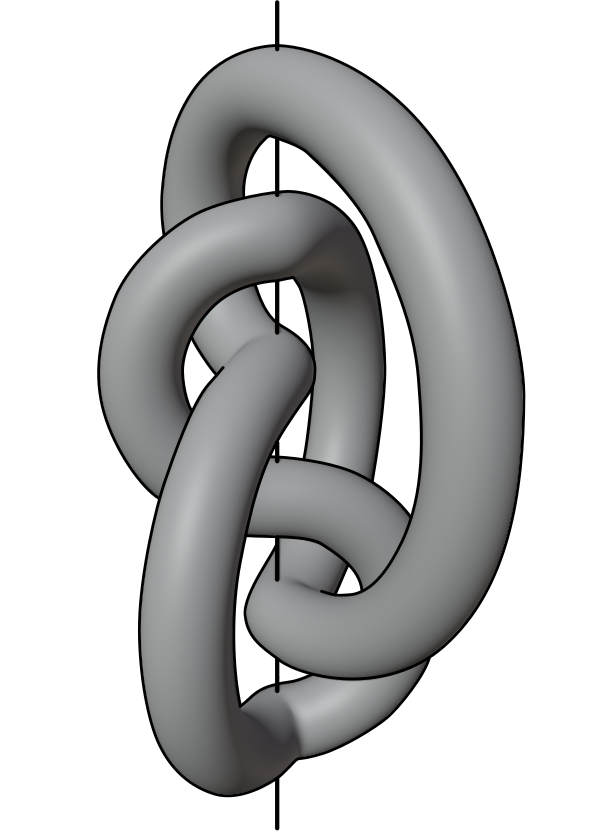}
    \caption{A torus in narrow position, following an arc presentation of the figure-eight knot of arc index 6, meeting the binding circle at 12 points.}
    \label{fig:3DNarrowPosition}
\end{figure}

The specific form we desire is one in which the layout of a page alternates between consisting entirely of shaded semicircles and consisting of many shaded semicircles and a single shaded ``band" which connects two shaded segments. Since our collection of tori has standard foliation, if this is the case, each shaded segment will be involved in exactly two bands, and we can take a collection of core curves of the $V_i$ to be an arc presentation $L_C$ of the companion link whose intersection with $S_\phi^1$ has exactly one point in each shaded segment.

The reason we consider this configuration perfect and want to achieve it is because it has the property that if $L$ is an arc presentation of a link lying in $V$, then (after an easy isotopy of $L$) the associated grid diagram is a satellite diagram that follows a grid diagram of the companion link (specifically the grid diagram corresponding to $L_C$).

This is the case because if we isotope $(L,T)$ (or deform our $\phi$ and $\theta$ coordinate system) so that the shaded segments are very short, and the bands exist for very short times, then the arc presentation $L$ has the property that any vertex of $L$ must have a $\phi$ coordinate very close to that of a vertex of $L_C$, and any arc of $L$ passing between vertices of $L$ that are close to different vertices of $L_C$ must have a $\theta$ coordinate very close to that of the appropriate arc of $L_C$. Further, if we push all arcs of $L$ that occur inside a semicircle through time so that they occur at a time when the segment is part of a band instead, we get a grid diagram of $L$ such that every corner of the diagram is very close to a corner of the diagram of $L_C$. Thus the grid diagram is a satellite diagram.

Therefore, to prove Theorem \ref{thm:satelliteDiagram}, it suffices to show how $(L,T)$ may be isotoped so that $T$ is in narrow position while maintaining the bound on the arc index of $L$.

\subsection{An easy simplification}
\label{subsec:emptySemis}

If some segment of $S_\phi^1$ has a semicircle but also contains no points of the link $L$, then by isotoping $T$ across (a regular neighbourhood of) a semicircular region disjoint from $L$ which connects it to $S_\phi^1$, we reduce the binding weight of $T$ by two. Such a semicircular region can be found living in any page during the time when a semicircle arc is joining the endpoints of the segment. The isotopy leaves $L$ unchanged.

In particular, since unshaded segments contain no points of $L$, the presence of any unshaded semicircle implies that such a simplifying move is available.

Thus, by applying this move (and other previously mentioned moves) as much as possible, we may assume that there are no unshaded semicircles, and that no shaded semicircles are ``empty" in the sense that their segment contains no points of $L$.

\subsection{Kazantsev's example}

Unfortunately, even if we assume that $T$ has standard foliation and is not susceptible to the moves discussed in Section \ref{subsec:emptySemis}, this does not imply that $T$ is in narrow position. The following example (in which $L = K$ is a knot and $T$ has one component), found by Kazantsev, is detailed in \cite{kazantsev}.

The torus has standard foliation, with 22 vertices and 22 saddles. A diagram of its pages is given in Figure \ref{fig:kazantsev}. On this diagram we include the arcs of the knot $K$ inside $V$. The torus $T$ is knotted in a trefoil, and $K$ is a cable of the trefoil with arc index 11.

\begin{figure}
    \centering
    \includegraphics[width=0.9\linewidth]{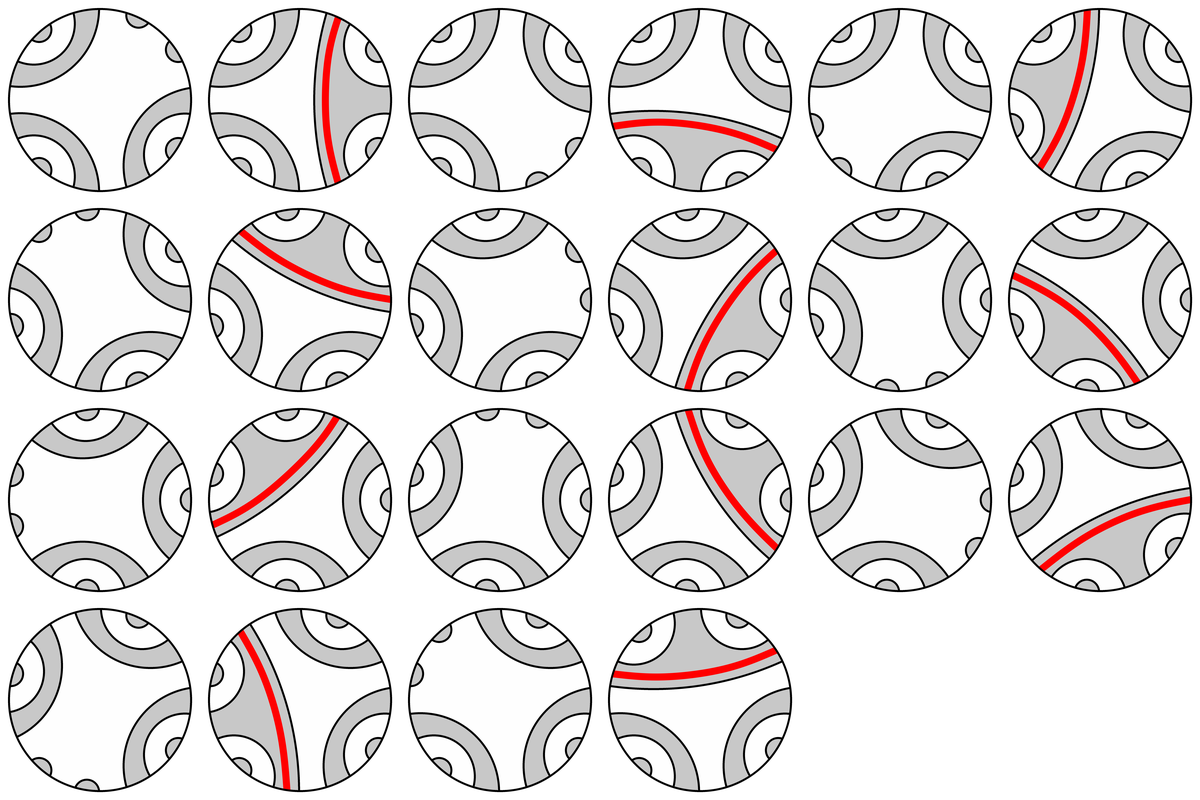}
    \caption{Kazentsev's example of a satellite torus that is not in narrow position but is resistant to being simplified without increasing the arc index of $K$. Arcs of $K$ are shown in red.}
    \label{fig:kazantsev}
\end{figure}

From the diagram it can be seen that there are no unshaded semicircles, and every shaded segment that has a semicircle (which is all of them) is visited by $K$. Nevertheless, $T$ is not in narrow position. Indeed, the grid diagram corresponding to $K$ is not a satellite diagram.

Not only that, but this grid diagram cannot be transformed by elementary moves into a satellite diagram without raising its arc index during the process (in fact the only moves possible without stabilising are cyclic permutations). All simplifying moves we have seen so far modify $K$ by some sequence of elementary moves that involves no stabilisations, so they are not up to the task of simplifying this $(K,T)$ into narrow position.

A vital part of our argument is the existence of a method of simplifying which is available even in dire situations like this, though it may incur great cost to the arc index of $L$. This method is described in Section \ref{sec:noRepeatedArcs}. The method will not put $T$ directly into narrow position, but will reduce it to a form which we will see in Section \ref{sec:repeatedArcs} poses no issue, and can be simplified via other methods which do not harm the arc index of $L$.

\section{Tori with repeated arcs}
\label{sec:repeatedArcs}

Some admissible tori with a standard foliation have \emph{repeated arcs}, by which we mean there exists an arc between two vertices that is formed at some time $t_1$, then subsequently disappears at some time $t_2$, but which then appears and disappears \emph{again} at times $t_3$ and $t_4$ before we return to the page at $t_1$. This is equivalent to the existence of a double edge between two vertices of the foliation (that is, there are two square tiles of $\mathcal{F}$ with both vertices in their boundary). This is a desirable property for a torus to have, as this property is possessed by every component of a collection of tori in narrow position, which is what we hope to eventually reach.

Throughout this section we make use of the graph of vertices of $\mathcal{F}$, saying that two vertices are connected by an edge if they are incident to a common tile, or equivalently if an arc between them is present at some time.

In this section we will see that if $T$ has standard foliation and is such that every component has repeated arcs, then if $T$ is not already in narrow position, we may simplify it into narrow position at no cost to the arc index of the satellite link inside of it.

First we will examine each torus individually and decide that since its foliation is standard and it has repeated arcs, then in fact it already has the property that every shaded segment has a repeated semicircle. It will follow that the union of the tori can be isotoped into narrow position simply by interchanging events.

\begin{lemma}
If an admissible torus with a standard foliation and $2n$ vertices has a repeated arc, then in fact every vertex is the endpoint of some repeated arc. If $n>2$, then every vertex is the endpoint of \emph{exactly one} repeated arc.
\end{lemma}
\begin{proof}
Suppose there is a repeated arc $\alpha$ with endpoints $v$ and $w$, and some vertex $a$ other than $w$ is adjacent to $v$ in the foliation. We will show that $a$ is also an endpoint of a repeated arc. Thus the property of being an endpoint of a repeated arc spreads to neighbours in the foliation. Since our torus is connected, if one vertex has this property (as we hypothesise), then all vertices do.

In the foliation $\mathcal{F}$ the vertex $v$ has four edges emerging from it. Two of these lead to $w$. It cannot be the case that these two edges are adjacent at $v$, because this would indicate that the arc $vw$ is involved in an event but is still present immediately afterwards, which is not possible. Thus the two edges leading to $w$ must be opposite each other at $v$, and similarly they must be opposite each other at $w$. So we have the configuration shown in Figure \ref{fig:vwFoliation}.

\begin{figure}
    \centering
    \includegraphics[width=0.4\linewidth]{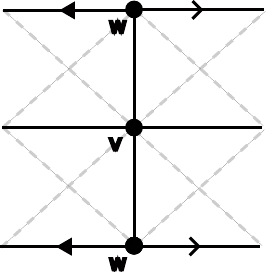}
    \caption{The region surrounding $v$ and $w$ in $\mathcal{F}$. Edges between vertices are in bold, separatrices are dotted and cross at saddle points.}
    \label{fig:vwFoliation}
\end{figure}

The vertex $a$ is the other endpoint of one of the horizontal edges from $v$. Without loss of generality we can take it to be the one on the right. We can see that two opposite edges emerging from $a$ both lead to the same vertex, call it $b$, which is adjacent in $\mathcal{F}$ to $w$. Thus $a$ is an endpoint of a repeated arc, and we have the result.

Finally, note that if the two edges emerging from $v$ that did not lead to $w$ went to the same vertex, then the entire foliation is only $4$ vertices. So if $n>2$, this does not occur, and $v$ is the endpoint of only one repeated arc. This is true of every other vertex as well by the same argument.
\end{proof}

It should be noted that $n \leq 2$ is not a scenario we are concerned with, as such an admissible torus cannot be knotted.

By repeating the argument from the preceding proof, we see that $\mathcal{F}$ has the form shown in Figure \ref{fig:fullFoliation}.

\begin{figure}
    \centering
    \includegraphics[width=0.8\linewidth]{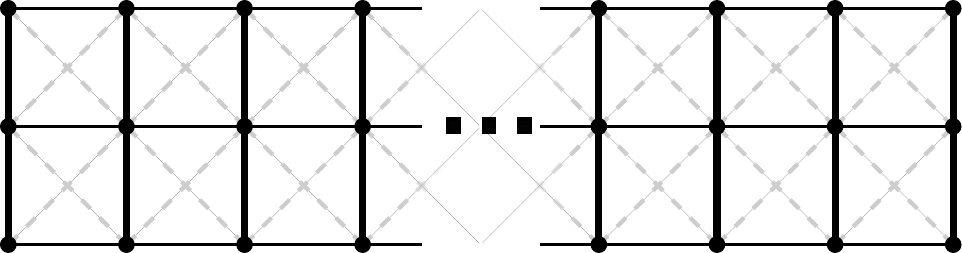}
    \caption{All of $\mathcal{F}$. The repeated arcs are vertical and are bolded. The upper and lower sides are identified by a vertical translation. The left and right circles are identified via a horizontal translation (followed by a half-twist if $n$ is odd).}
    \label{fig:fullFoliation}
\end{figure}

From this we can conclude the following:

\begin{itemize}
    \item There are $n$ repeated arcs, joining the vertices together in pairs.
    \item The repeated arcs form a loop. That is, we can label them $\alpha_0$, $\alpha_1$, ..., $\alpha_{n-1}$, $\alpha_n = \alpha_0$ in the order that they are encountered when travelling around the torus the ``long way". Then we can see from the foliation that every event is either the arcs $\alpha_i$ and $\alpha_{i+1}$ joining together, or the arcs $\alpha_i$ and $\alpha_{i+1}$ separating from each other.
\end{itemize}

Since every pair of adjacent saddles in $\mathcal{F}$ has opposite orientations, (since the foliation is standard), then for each $i$, the arcs $\alpha_{i-1}$ and $\alpha_{i+1}$ are on the same side of $\alpha_i$, and we call this the \emph{interactive side of $\alpha_i$} (because it is the side on which events involving $\alpha_i$ occur). This is because if the arc $\alpha_i$ appears after a negative (resp. positive) event, the next event involving it will be positive (resp. negative) and the signs of these events tell us which side of $\alpha_i$ the event will occur on.

The following terminology will be useful for the next proofs:

\begin{itemize}
    \item We say a repeated arc $\alpha$ is \emph{nested inside} another repeated arc $\beta$ if $\alpha$ is contained on the non-interactive side of $\beta$.
    \item We say a repeated arc $\alpha$ \emph{interleaves} with another repeated arc $\beta$ if the two endpoints of $\alpha$ are on opposite sides of $\beta$ (or equivalently, the two endpoints of $\beta$ are on opposite sides of $\alpha$).
    \item We say a repeated arc $\alpha$ \emph{interacts towards} another repeated arc $\beta$ if $\beta$ is on the interactive side of $\alpha$.
\end{itemize}

\pagebreak
\begin{lemma}
    Suppose we have a knotted admissible torus with standard foliation that has repeated arcs, so that all of the above observations apply. Now suppose that there exist two interleaving repeated arcs $ab$ and $cd$ so that we have (possibly a mirror image of) the configuration on the left of Figure \ref{fig:interleave}. Then there are repeated arcs $a_+b_+$ and $c_+d_+$ that interact with $ab$ and $cd$ respectively such that when we traverse the binding circle in some direction starting at $a$, we encounter $a,c_+,c,b,b_+,d,d_+,a_+$ in that order. That is, we have (possibly a mirror image of) the configuration shown on the right of Figure \ref{fig:interleave}.
\end{lemma}
\begin{figure}
    \centering
    \includegraphics[width=0.8\linewidth]{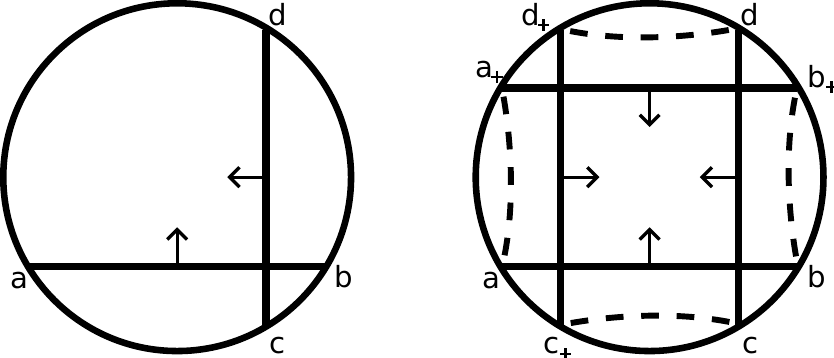}
    \caption{If the configuration on the left is present, then the configuration on the right is too. Repeated arcs are solid, while non-repeated arcs are dotted. Each repeated arc has an arrow pointing to its interactive side.}
    \label{fig:interleave}
\end{figure}
\begin{proof}
    Since the arc $ab$ is repeated, it is present during two disjoint time intervals $(t_1,t_2)$ and $(t_3,t_4)$. Similarly, $cd$ is present during two disjoint time intervals $(s_1,s_2)$ and $(s_3,s_4)$. Since $ab$ and $cd$ interleave, they cannot coexist at any time, so all four of these intervals are pairwise disjoint. Without loss of generality (by choosing a start time, and potentially swapping the labels of the two intervals for an arc), we have either $t_1<t_2<s_1<s_2<t_3<t_4<s_3<s_4$ in the case that $cd$ appears once during each period of $ab$'s absence, or $t_1<t_2<t_3<t_4<s_1<s_2<s_3<s_4$ in the case that $cd$ appears twice during a single period of $ab$'s absence.

    Suppose for a contradiction that the first case occurs. Consider a loop formed by the arc from $a$ to $b$ at some time in $(t_1,t_2)$ followed by the arc from $b$ to $a$ at some time in $(t_3,t_4)$. We see from the foliation that this curve is non-trivial on the torus, but it bounds a disc in $S^3$ (the disc made from a half disc in each of the two pages containing these arcs). Thus it is a meridian curve, since otherwise the curve, regarded as a knot, would be either equivalent to the core curve or a cable of the core curve, and in either case it could not bound a disc, because the torus is knotted. Similarly by taking an arc from $c$ to $d$ at some time in $(s_1,s_2)$ and another from $d$ to $c$ at some time in $(s_3,s_4)$, we obtain another meridian curve. But since we assumed $t_1<t_2<s_1<s_2<t_3<t_4<s_3<s_4$, these two meridian curves are linked. This is impossible, since if we take a meridian disc properly embedded in the solid torus bounded by one of these curves, the other curve would have to pierce the interior of this disc despite being contained in the torus itself while the interior of the disc is disjoint from the torus. Thus the first situation does not occur, and it is the case that $t_1<t_2<t_3<t_4<s_1<s_2<s_3<s_4$.

    During $(t_4,t_1)$, the repeated arc $ab$ is interacting with some other repeated arc, call its vertices $a_+$ and $b_+$ (so that $a$ joins with $a_+$ and $b$ joins with $b_+$). Call the other arc that $ab$ interacts with (during $(t_2,t_3)$) $a_-b_-$ in the same way. Similarly, let $c_+d_+$ be the arc that $cd$ is interacting with during $(s_4,s_1)$, and $c_-d_-$ be the other arc that $cd$ interacts with.

    The following arcs cannot interleave with $aa_+$ or $bb_+$ since they coexist at some time: $cc_+$, $dd_+$ ,$cd$, $c_+d_+$ ,$c_-d_-$, $cc_-$, and $dd_-$.
    Similarly, none of $aa_+$, $bb_+$, $ab$, $a_+b_+$, $a_-b_-$, $aa_-$, and $bb_-$ can interleave with $cc_+$ or $dd_+$, since they coexist at some time.

    These restrictions, taken together, force the configuration on the right side of Figure \ref{fig:interleave} to appear (or a mirror image of this configuration, if a mirrored version of the left-hand side was present).
\end{proof}

\begin{lemma}
    \label{lem:repeatedArcToriAreNice}
    Suppose again we have a knotted admissible torus with standard foliation that has repeated arcs, so that all of the above observations apply. Then every repeated arc is a semicircle.
\end{lemma}
\begin{proof}
    Suppose for a contradiction that there are repeated arcs that interleave with each other or are nested inside each other. As before, say the torus has $2n$ vertices, and label the repeated arcs $\alpha_0$, $\alpha_1$, ..., $\alpha_{n-1}$, $\alpha_n = \alpha_0$. For simplicity of notation we will understand $\alpha_i$ to mean $\alpha_r$ where $r$ is the residue of $i$ mod $n$, for any integer $i$.
    
    Let $m \in \mathbb{N}_{>0}$ be minimal such that there exists $i \in \mathbb{Z}$ with $\alpha_{i+m}$ either interleaving with $\alpha_i$ or nested inside $\alpha_i$. Such an $m < n$ exists due to our initial assumption.
    
    It is important to note that $m \geq 2$, since an arc cannot interact directly with one that interleaves it (as the two arcs can never coexist), nor with one that is nested inside of it (by the definition of nesting, the nested arc is on the non-interactive side). Also $m \leq n-2$ for similar reasons. In what follows we will use implicitly that $\alpha_{i+m-1}$ is distinct from $\alpha_i$.
    
    \textbf{Claim 1:} There exists $p \in \mathbb{Z}$ such that $\alpha_{p+m}$ is nested inside $\alpha_p$.

    \textbf{Proof of claim 1:} Let $i$ be as in our definition of $m$. If $\alpha_{i+m}$ is nested in $\alpha_i$, then we are already done, since $p=i$ fits our requirement.
    
    Otherwise, $\alpha_i$ and $\alpha_{i+m}$ interleave. By the previous lemma, we can label the endpoints of $\alpha_i$ $a$ and $b$, and the endpoints of $\alpha_{i+m}$ $c$ and $d$ such that the configuration on the right of Figure \ref{fig:interleave} is present (or, perhaps, a mirror image of it). The arc $c_+d_+$ is either $\alpha_{i+m-1}$ or $\alpha_{i+m+1}$. But $c_+d_+$ interleaves with $ab$, so by minimality of $m$, $\alpha_{i+m-1}$ cannot be $c_+d_+$. Therefore $\alpha_{i+m+1} = c_+d_+$. For similar reasons, $a_+b_+$ is $\alpha_{i-1}$.
    
    Now consider the position of $\alpha_{i+m-1} = c_-d_-$. It cannot interleave with $aa_+$ (since they coexist at some point). It must be contained on the side of $aa_+$ which contains $cd$ (since the arcs $cc_-$ and $dd_-$ also coexist with $aa_+$). It must be on the interactive side of $cd$ (since $cd$ interacts with it). It cannot interleave with $ab$, by minimality of $m$. Therefore either both endpoints of $\alpha_{i+m-1}$ lie on the portion of the binding circle moving counter-clockwise from $a$ to $c$, or both endpoints lie on the portion moving counter-clockwise from $d$ to $a_+$. In the first case, $\alpha_{i+m-1}$ is nested inside $ab = \alpha_i$, contradicting minimality of $m$. Therefore the second case occurs, and $\alpha_{i+m-1}$ is nested inside $a_+b_+ = \alpha_{i-1}$. Setting $p = i-1$, the claim is proved. $\square$

    \textbf{Claim 2:} If $\alpha_{p+m}$ is nested inside $\alpha_p$, then $\alpha_{p+m}$ interacts towards $\alpha_p$ and also $\alpha_{p+m-1}$ is nested inside $\alpha_{p-1}$.

    \textbf{Proof of claim 2:} First note that if $\alpha_{p+m}$ interacts away from $\alpha_p$, then $\alpha_{p+m-1}$ is also nested inside $\alpha_p$, which contradicts minimality of $m$. Therefore $\alpha_{p+m}$ interacts towards $\alpha_p$.

    Now consider the position of $\alpha_{p+m-1}$. If $\alpha_{p+m}$ interacts with $\alpha_{p+m-1}$ while $\alpha_p$ is present, then $\alpha_{p+m-1}$ is also nested inside $\alpha_p$, contradicting minimality of $m$. So $\alpha_{p+m}$ interacts with $\alpha_{p+m-1}$ while $\alpha_p$ is joined to one of its neighbours $\alpha_k$, with $k \in \{p-1,p+1\}$.

    Thus each endpoint of $\alpha_{p+m-1}$ lies either on the non-interactive side of $\alpha_p$, or on the non-interactive side of $\alpha_k$. By minimality of $m$, $\alpha_{p+m-1}$ cannot be nested inside $\alpha_p$, and it cannot interleave with $\alpha_p$. Therefore it is nested inside (and is distinct from) $\alpha_k$. But then $k=p+1$ is impossible (if $m > 2$ this is because it would contradict minimality of $m$, while if $m=2$, this would contradict our observation that $\alpha_{p+m-1}$ and $\alpha_k$ are distinct), so $k=p-1$, $\alpha_{p+m-1}$ is nested inside $\alpha_{p-1}$, and the claim is proved. $\square$

    So by claim 1 we have a $p$ satisfying the hypothesis of claim 2. But then the conclusion of claim 2 means that $p-1$ also satisfies the hypotheses of claim 2, and so on forever. For every $j \in \mathbb{N}$, $\alpha_{p-j+m}$ is nested inside $\alpha_{p-j}$ and interacts towards it.

    Every repeated arc is nested inside another one towards which it interacts. Thus there is an infinite sequence of repeated arcs, each one nested inside and interacting towards the next. But this is not possible, since then each arc in this sequence must have strictly more vertices on its non-interactive side than its predecessor does. We get an infinite strictly increasing sequence of non-negative integers, and all of them are at most $2n$. This is a contradiction, so our original assumption was false.

    No pair of repeated arcs are nested or interleaving. If there is a vertex on the non-interactive side of a repeated arc $ab$, then that vertex is the endpoint of a repeated arc which must either nest inside $ab$ or interleave $ab$. Neither of these happen, so there can be no vertices on the non-interactive side of a repeated arc. All repeated arcs are semicircles.
\end{proof}

Due to these lemmas, each component of $T$ that isn't unknotted and has repeated arcs is positioned very nicely. If every component of $T$ has repeated arcs, then the whole of $T$ is positioned very nicely. It doesn't quite guarantee narrow position, but this can easily be fixed.

\begin{proposition}
    \label{prop:toNarrowPosition}
    Suppose we have an admissible collection of satellite tori $T$ for an arc presentation $L$ of a link, such that $T$ has standard foliation, no component of $T$ is unknotted, and all components of $T$ have repeated arcs.

    Then we can perform an ambient isotopy of the 3-sphere such that $T$ is put into narrow position without increasing the arc index of $L$.
\end{proposition}
\begin{proof}
    First, by Lemma \ref{lem:repeatedArcToriAreNice}, all repeated arcs are semicircles when we only consider one torus. We would like to be sure than no other components of $T$ have vertices inside these semicircles. Considering only this torus, the space swept out by the semicircular regions and the bands between them is a solid torus, so since this torus is knotted, this must be the shaded side. Since all the $V_i$ are disjoint, no vertices of $T$ lie inside the semicircles, and so the repeated semicircles truly are semicircles, even when we look at all of $T$ together.

    Observe from the foliation that every event is either two shaded semicircles joining together to form a ``band", or a band collapsing to form two semicircles. Suppose there is some band that forms but does not collapse immediately afterwards. That is, other events of $T$ happen while it is connected. Each of these events occurs completely on one side or the other of this band. There may also be some events in which arcs of the link pass through this band.

    There is no obstruction to bringing every event in which an arc of the link passes through this band (or just goes from one of the two relevant segments to itself while the band is present) as well as the band collapse event forward in time so that they all happen in a row immediately after the band is formed. The band now collapses immediately after forming. No pairs of temporally-adjacent events have been separated, except some pairs involving the formation and collapse events for the band in question, and so any other band which collapsed immediately after forming before still does so after this isotopy.

    This procedure reduces the number of bands that do not collapse immediately after they form by one. By repeating this until no such bands remain, we place $T$ into narrow position. At no point are any new arcs introduced to the link.
\end{proof}

Now the only situations in which we have not seen a way to simplify are those where the foliation is standard but some component of $T$ has no repeated arcs. In Section \ref{sec:noRepeatedArcs} we will see how such $T$ can be simplified, not directly to narrow position, but to a position where the assumptions of Proposition \ref{prop:toNarrowPosition} are satisfied, at which point this result handles the rest of the process. This upcoming technique, unlike the ones we have seen so far, damages the arc index of $L$ significantly, though we will be able to control exactly how much damage is done.

\section{Tori without repeated arcs}
\label{sec:noRepeatedArcs}

There remain situations in which none of the simplifying techniques we have described are available. That is, situations where $T$ has standard foliation but at least one component has no repeated arcs, such as Kazantsev's example (Figure \ref{fig:kazantsev}).

What follows is a description of an isotopy of $(L,T)$ in $S^3$ that we can perform even in such situations, after which $L$ is in an arc presentation and $T$ has standard foliation with every component having repeated arcs, so that we can then conclude according to Proposition \ref{prop:toNarrowPosition}.

The new arc index of our presentation for $L$ could be significantly larger than the original arc index $\alpha$, but this new arc index can be controlled in terms of $\alpha$ and the binding weight of $T$.

Thus all that will remain in order to prove Theorem \ref{thm:satelliteDiagram} is establishing a bound on the binding weight of $T$ in terms of $\alpha$, which is accomplished in Section \ref{sec:bindingWeight}.

The isotopy will happen in two steps. First, we isotope $(L,T)$ so that every $T_i$ has a repeated arc which forms a meridian curve for $T_i$, though the foliation may be far from standard. Then we apply the simplifying moves and argue that the presence of a repeated arc on every torus is preserved.

\subsection{Shrinking a meridian disc}
\label{subsec:shrinkDisc}

The first part of the isotopy will be supported in a small neighbourhood of each $V_i$, so here we limit our view to a single $V_i$ and the components $L_i$ of $L$ lying inside it. We describe how we may isotope $L_i$ and $T_i$ so that $T_i$ has a meridian curve consisting of a repeated arc, and then we simply run this process individually on each $V_i$ (except those that already have a repeated arc).

Our approach is to first find a meridian curve of $T_i$ that lies in some pair of pages, and then isotope this across most of a meridian disc so that it consists of only two arcs. Many arcs of $L_i$ may need to be redirected so that they do not meet the new position of $T_i$.

The following lemma is our starting point. It is very similar to Theorem 2.2 of \cite{Plachta2006EssentialTA}. That Theorem ensures that we can take our two pages to be opposite, while this one allows us to choose one of the pages freely (a freedom we will not need for our purposes), and admits a shorter proof.

\begin{lemma}
    \label{lem:twoPageMeridian}
    Suppose $T$ is an admissible torus whose foliation $\mathcal{F}$ contains no closed leaves or poles. Let $D_t$ be a page not containing a singularity of $\mathcal{F}$. Then there is another non-singular page $D_s$ such that some component of $T \cap (D_t \cup D_s)$ is a curve bounding a meridian disc of a solid torus component of $S^3 \setminus T$.
\end{lemma}
\begin{proof}
    After reparametrising $\theta$, we may take $D_t$ to be $D_0$ and the $n$ events, which are all saddles, to be located at $\theta = \frac{1}{2n}\cdot2\pi, \frac{3}{2n}\cdot2\pi, ..., \frac{2n-1}{2n}\cdot2\pi$ so that the values $\theta = 0, \frac{2\pi}{n}, 2 \cdot \frac{2\pi}{n}, ..., (n-1) \cdot \frac{2\pi}{n}$ give one page in each interval between events.

    Consider the sequence of subsurfaces $S_1 \subseteq S_2 \subseteq ... \subseteq S_{n-1} \subseteq T$ given by $S_k = T \cap \{0\leq\theta\leq k \cdot \frac{2\pi}{n}\}$.

    Then $S_1$ is a union of disjoint discs in $T$, and $S_{n-1}$ is the complement of a union of disjoint discs in T. Each $S_{i+1}$ is obtained from $S_i$ by the addition of a band in $T$ running between two (not necessarily distinct) boundary components of $S_i$,
    followed by an ambient isotopy. The band is a patch of $T$ surrounding the saddle point that occurs at $\theta = \frac{2i+1}{2n}\cdot2\pi$.

    Consider, for each $S_i$, whether $S_i$ is contained in a union of discs whose boundaries are circles of $\partial S_i$ (two equivalent formulations of this property are $S_i$ being a union of discs in $T$ that have had subdiscs removed, or all curves of $\partial S_i$ being trivial and the outermost ones bounding discs on their $S_i$ side). This is a property possessed by $S_1$ but not $S_{n-1}$, so there is some $i$ such that $S_i$ has this property and $S_{i+1}$ does not. Consider the band which is attached to $S_i$ to form $S_{i+1}$.

    If the band lies inside a disc bounded by a curve of $\partial S_i$ then $S_{i+1}$ would also have the property, so the new band must connect two outermost circles of $\partial S_i$, or connect an outermost circle to itself.

    In the first case, $S_{i+1}$ would still have the property, so the band must connect an outermost circle to itself. In this case $S_{i+1}$ has two ``new" boundary components, and if either of them bounds a disc on either side, $S_{i+1}$ would have the property. Thus both of these new curves are homotopically non-trivial.

    These homotopically non-trivial curves are contained in $\partial S_{i+1}$ which lies in the pair of pages $D_0 \cup D_{(i+1)\cdot \frac{2\pi}{n}}$. These pages together form a 2-sphere $P$, which intersects $T$ in several circles. We have established that at least one of these circles in non-trivial on $T$. Let $\gamma$ be such a circle that is innermost (on $P$) among all such circles. This circle bounds a disc $D$ in $P$ that intersects $T$ only in circles that bound discs in $T$. Delete the subdiscs bounded by the outermost of these circles in $D$ and replace them with the discs that these circles bound on $T$. Finally we push the resulting disc off $T$ slightly, into the component of $S^3 \setminus T$ that $D$ is contained in near its boundary. This yields a disc disjoint from $T$ except at its boundary which is a non-trivial curve in $T$. Thus this is a meridian disc, and $\gamma$ is a meridian curve lying in $P$.
\end{proof}

Since $T_i$ is not unknotted, the above Lemma tells us that we can find a meridian curve of $T_i$ in a pair of pages that bounds a meridian disc in $V_i$. By reparametrising $\theta$, we may assume that $T_i$ has a meridian curve $\gamma$ lying in $P = D_0 \cup D_\pi$ (and we take $\gamma$ to be innermost among such curves). For the rest of this section we imagine $S^1_\phi$ to be a vertical line in the plane $P$, the space above $P$ to be the points with $\theta \in (0,\pi)$, and the space below $P$ to be
the points with $\theta \in (\pi,2\pi)$.

Now $\gamma$ bounds a disc $A$ in $V_i$ that lies in $P$ except where it runs very close to $T$ across discs bounded by other curves in $P$. Choose a segment of $S^1_\phi \cap V_i$ that is incident to $\gamma$, and let $\beta$ be an arc in $P$ in a neighbourhood of this segment that has both ends on $\gamma$ and encloses between $\beta$ and $\gamma$ a disc $B$ containing a portion of this segment that contains all points of $L$ that lie on this segment. Let $\delta = \gamma \setminus B$. Then $C = \overline{A \setminus B}$ is a disc with boundary $\beta \cup \delta$. An illustration of what the relevant part of $P$ may look like, together with the arc $\beta$, can be seen in Figure \ref{fig:genericMeridianBeta}.

\begin{figure}
    \centering
    \includegraphics[width=0.5\linewidth]{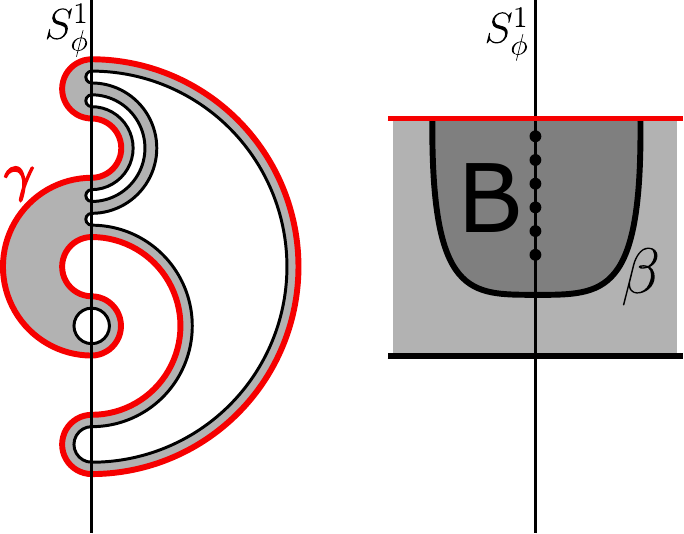}
    \caption{On the left is an example of how $\gamma$ might lie in $P$, together with the outermost circles in the disc it bounds. The disc $A$ lies in the shaded region except where it runs near $T_i$. On the right is the arc $\beta$ and the disc $B$ that it bounds with $\gamma$.}
    \label{fig:genericMeridianBeta}
\end{figure}

Our goal now is to perform a small perturbation of $C$ to form a disc $D$, then isotope $L_i$ in $V_i$ so that it is disjoint from $D$. All arcs that pass through the meridian disc $B \cup D$ will need to be diverted through $B$ instead of $D$.

Once this is accomplished, we can isotope $T_i$ across a neighbourhood of $D$  to take $\gamma$ to $\partial B$ so that afterwards $B$ is a meridian disc of $V_i$ and the meridian curve $\partial B$ consists of only two arcs.

In order to perturb $C$ suitably, we first examine the main component of $C \cap P$, which is a disc with holes. It is the shaded region from the left of Figure \ref{fig:genericMeridianBeta}, with $B$ and a small region near the boundary components other than $\gamma$ excluded. This region has one boundary component $\beta \cup \delta$ and potentially several others. The binding circle $S^1_\phi$ intersects it in a collection of disjoint properly embedded arcs which may run between distinct boundary components or from a boundary component to itself.
Exactly one of these arcs emerges from $\beta$.

We choose a collection of ``green" arcs in this region as follows: one arc runs just parallel to $\delta$, ending on $\beta$ at both ends. For each other boundary component, we add an arc running from $\beta$ to near that boundary component, running close to it around the circle, then back to $\beta$ running parallel to the beginning of the arc. These arcs can be chosen disjoint, and they split the region into two parts. A ``red" part in the thin spaces between the green arcs and the boundary components, and a``blue" part everywhere else. An example is shown in Figure \ref{fig:greenArcs}.

\begin{figure}
    \centering
    \includegraphics[width=0.5\linewidth]{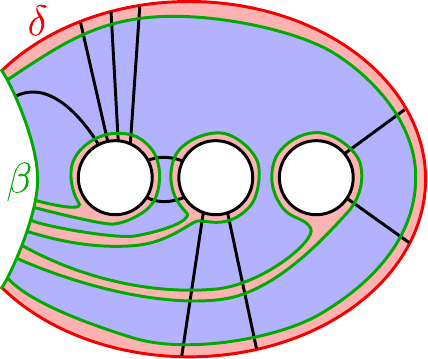}
    \caption{An example of how we choose our green arcs. All of them end at $\beta$ at both ends, and every segment of the binding circle (shown in black) meets a green arc.}
    \label{fig:greenArcs}
\end{figure}

We may choose these arcs so that on any particular segment of the binding circle, every point of $L$ is contained in a single contiguous blue portion of that segment, and each green arc meets each segment of $S^1_\phi$ at most twice. It will be vitally important for redirecting $L$ that every segment meets at least one green arc, and that this arc ends at $\beta$.

Take $C$ and perturb it by pushing the blue parts upwards into $\{0 < \theta < \pi\}$ and pushing the red parts downwards into $\{\pi < \theta < 2\pi\}$. This forms our disc $D$ which intersects $P$ in $\beta$ and $\delta$ and the green arcs (as well as some curves very close to $T_i$ which will be of no concern).

\begin{proposition}
    \label{prop:redirect}
    With $D$ as described above, it is possible to isotope $L_i$ in $V_i$ so that afterwards it is arc presented and is disjoint from $D$.
\end{proposition}
\begin{proof}
    First observe that arcs of $L_i$ are disjoint (or can be made disjoint by isotopy in their page) from $D$ except for those that are incident to one of the segments met by $D$ (except for the segment $s$ that meets $\beta$) and are also incident to some other segment, and have a $\theta$ coordinate in $(0,\pi)$. These arcs pierce the blue part of $D$, and these are the ``problematic" arcs that we need to redirect through $B$. Note that such an arc could be problematic at only one end or at both (if it is incident at both ends to segments meeting $D$).

    Also of note are those arcs that are incident to one of these segments and have a $\theta$ coordinate in $(0,\pi)$, but are incident to the same segment at both ends. These do not meet $D$, but will need to be moved anyway since they would obstruct our other operations.

    We focus on each segment meeting $D$ (except $s$) in turn and explain how to redirect the problematic arcs.

    At each segment, we denote by $c$ the contiguous blue portion of the segment to which all the original arcs of $L$ meeting this segment are incident. We select a green point (a point of intersection between this segment and a green arc) which is one of the endpoints of $c$.

    We also choose whether we will be ``flattening" the arcs towards $\theta = 0$ or $\theta = \pi$. We consider the arcs incident to $c$ that have $\theta$ coordinates in $(0,\pi)$, and among these we focus on the one whose $\theta$ coordinate is closest to $0$ if we're flattening towards $0$, or to $\pi$ if we're flattening towards $\pi$.

    If this arc is incident at both ends to this segment, we simply push it through time beyond $\theta = 0$ (or $\pi$) so that it no longer lies above $P$. There is no obstruction to doing this; any arc that interleaves this arc will also be incident to $c$, and will not lie between our arc and $\theta=0$ (or $\pi$) since we chose a closest arc in $\theta \in (0,\pi)$. At this point we are done with this arc. It will never need to be moved again, and we proceed to the next step.

    \begin{figure}
        \centering
        \includegraphics[width=0.8\linewidth]{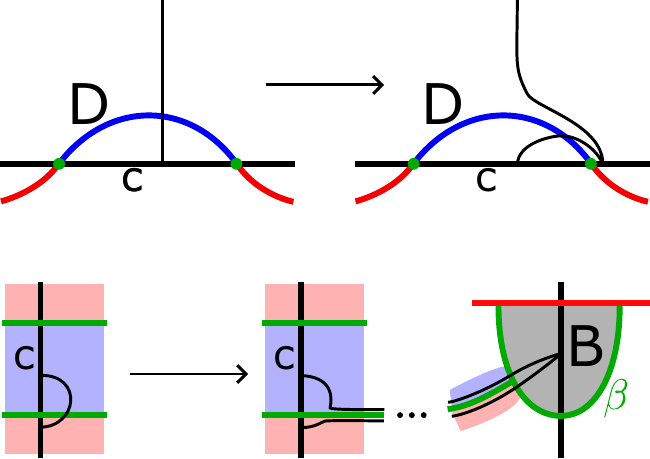}
        \caption{The top left shows a problematic arc, which pierces $D$. First we stabilise it near the green point as shown. After flattening the new arc into $P$, we let it follow a green arc all the way to $B$.}
        \label{fig:redirection}
    \end{figure}

    If instead our chosen arc is a problematic one, we first isotope it in its page to add another point of incidence to the segment just beyond our chosen green point (closer to that green point than any other point of $L_i$ on this segment). This change is depicted at the top of Figure \ref{fig:redirection}. The arc has become two arcs. One does not pierce $D$ near this segment (though it might pierce $D$ near another segment, if the original arc was problematic at both ends). The other does pierce $D$ near this segment but is incident to this segment at both ends. We push it through time to exactly $\theta = 0$ (or $\pi$). This is possible because any arc interleaving this one would need to be incident to $c$, and such an arc cannot come between our arc and $\theta=0$ (or $\pi$) since we picked a closest arc in $\theta \in (0,\pi)$.

    Once this arc lies in $P$, it meets a green arc. We isotope the arc in $P \setminus T_i$ so that instead of passing through the green arc, it follows it on its blue side all the way to $\beta$, enters $B$ and touches the segment $s$, then leaves $B$ and follows the green arc back on the other side (the red side) until it returns to the segment it started at. This is depicted at the bottom of Figure \ref{fig:redirection}.

    This replaces the single arc of $L_i$ with several arcs, since it may pass through several segments of $S^1_\phi$. Note that the new redirected path in $P$ intersects $D$ only at two points of $\beta$. We resolve this by taking each of our new arcs in $P$ and pushing them very slightly below $P$ if they run through a blue region or very slightly above $P$ if they run through a red region. As a result, our original problematic arc has been isotoped to no longer pierce $D$ at this end.

    In either case, the number of arcs incident to $c$ with $\theta$ coordinates in $(0,\pi)$ has decreased by one. We then repeat the process, dealing with the next closest arc to $\theta=0$ (or $\pi$), and so on. Eventually there are no arcs incident to $c$ with $\theta$ coordinate in $(0,\pi)$, and thus all problematic arcs at this segment have been made disjoint from $D$ near this segment. We move on to another segment and perform the same procedure there.

    Importantly, at no point do we introduce any new problematic arcs at any segments, so once we have performed the process at each segment meeting $D$ (except $s$), $L_i$ is disjoint from $D$ (and instead pierces $B$ many times).
\end{proof}

As described before, once this is accomplished, we isotope $T_i$ across a neighbourhood of $D$ so that $\partial B$ is a meridian curve of $T_i$.

We do the same process to every $V_i$, and once we're done, every component of $T$ has a meridian curve that has length 2 in the vertex graph of the foliation.

\subsection{The effect on the arc index}
\label{subsec:arcIndexIncrease}

The process described in the proof of Proposition \ref{prop:redirect} raises the arc index of our arc presentation, and we need control over the size of this increase. This will depend on certain choices we make.

Firstly, observe that there are at least two non-trivial curves of $T_i \cap P$ lying in $P$, and $P$ is a sphere, so there are at least two innermost such curves, call them $\gamma_1$ and $\gamma_2$. These are adjacent to distinct components of $V_i \cap P$, which we call $R_1$ and $R_2$ respectively (these are distinct since $R_1$ lies on the inner side of $\gamma_1$, and similarly $R_2$ lies on the inner side of $\gamma_2$). For $j \in \{1,2\}$, let $v_j$ be the number of vertices of $L_i$ that lie in $R_j$, and let $s_j$ be the number of shaded segments that are part of $R_j$.

Since there are $\alpha_i$ vertices of $L_i$ in $P$ (where $\alpha_i$ is the original arc index of $L_i$), and there are $w^{(i)}_\beta / 2$ segments of $V_i \cap S^1_\phi$ (where $w^{(i)}_\beta$ is the binding weight of $T_i$), we find that $v_1+v_2 \leq \alpha_i$ and $s_1+s_2 \leq w^{(i)}_\beta /2$.

Note that the smaller $v_js_j$ will always be at most $\alpha_i w^{(i)}_\beta /8$. To see this, write $v_j = x_j \alpha_i$ and $s_j = y_j (w^{(i)}_\beta / 2)$ so that $x_1+x_2 \leq 1$ and $y_1+y_2 \leq 1$. Then observe that if $x_1+y_1 \leq 1$ then $x_1y_1 \leq 1/4$, but otherwise $x_2+y_2 \leq 1$ and therefore $x_2y_2 \leq 1/4$.

We choose $\gamma$ to be whichever $\gamma_j$ has $v_js_j$ smaller. Then $\gamma$ borders a shaded region containing $v$ vertices of $L_i$ and $s$ segments, with $vs \leq \alpha_i w^{(i)}_\beta /8$. 

Secondly, when deciding which way to flatten the arcs at a segment, we should choose the direction in which the green arc passes through the fewest other segments before reaching $\beta$. Since each green arc meets each segment at most twice, each green arc meets segments at at most $2s$ points, one of which is our starting point. Thus by going the ``short" way, we can ensure that the arc encounters at most half of the $2s - 1$ or fewer remaining points before reaching $\beta$.

When we redirect a problematic arc, we add one arc initially, then add $2k+1$ arcs when we isotope it in $P$, where $k$ is the number of times the green arc passes through a segment before reaching $\beta$. By the above we can ensure $k \leq s - 1/2$ so that we add $2k+2 \leq 2s + 1$ arcs while redirecting this problematic arc.

Finally, note that there are $2v$ arc-ends (by which we mean the piece of an arc of $L_i$ close to one of its endpoints) that are incident to the shaded region, and each of these may pierce $D$. However, if we were to interchange the roles of the blue and red regions, pushing blue down and red up instead of the opposite, any arc-end that was previously problematic would no longer be problematic. Therefore if we choose whichever approach leads to fewer problematic arc-ends, then we will only need to redirect an arc at most $v$ times.

Thus overall we may accomplish the isotopy while raising the arc index of $L_i$ by at most $v(2s+1) = 2vs + v \leq \alpha_i w^{(i)}_\beta /4 + \alpha_i$.

We proceed as above for each torus individually, so once we have performed the isotopy for all of them, the arc index of $L$ has been raised by at most
\[
\sum_i \left(\frac{\alpha_iw_\beta^{(i)}}{4}+\alpha_i\right) \leq \left(\sum_i \alpha_i\right)\left(\frac{w_\beta}{4} + 1\right) = \frac{\alpha w_\beta}{4} + \alpha
\]
where $w_\beta$ is the binding weight of $T$ and $\alpha$ is the arc index of $L$ before we perform the isotopies.

In order to properly bound this quantity we need a bound on $w_\beta$ in terms of $\alpha$. This will be the goal of Section \ref{sec:bindingWeight}.

\subsection{Achieving narrow position}

The crux of this argument is that this is all of the necessary damage to the arc index of $L$. Once we have isotoped $(L,T)$ as described in \ref{subsec:shrinkDisc}, it is possible to further isotope it into narrow position without further increasing the arc index of $L$ at all.

In Proposition \ref{prop:toNarrowPosition} we established that this will be the case as long as $T$ is such that every component has a meridian of length 2 in the vertex graph of the foliation $\mathcal{F}$, and also $\mathcal{F}$ is a standard foliation.

Our isotopy has ensured that $T$ has the first property, but $\mathcal{F}$ may be far from standard. Since $T$ is an admissible collection of satellite tori for the arc-presented $L$, we can standardise $\mathcal{F}$ without increasing the arc index of $L$ by applying the simplifying moves described in Section \ref{subsec:simplifyingMoves}. After doing this, $T$ will have the other necessary property. All we need is for the first property to remain intact throughout this simplifying process, and then we are done by Proposition \ref{prop:toNarrowPosition}.

\begin{lemma}
    If $T$ is a collection of admissible tori with no poles or closed leaves such that each component of $T$ has a meridian of length 2 in the vertex graph of the foliation, then when the simplifying moves isotope $T$ to have a standard foliation, the resulting $T$ still has each component having a length 2 meridian.
\end{lemma}
\begin{proof}
    First note that it is enough to establish that the minimal vertex-graph-lengths of the meridians of the components of $T$ do not increase. They will never be zero since they are homotopically non-trivial, and can never be one since the vertex graph is bipartite. Thus if their length cannot increase, it will remain 2 throughout.

    Recall that in Section \ref{subsec:foliationEffect} we observed that when we have a vertex of valence less than four and apply the relevant simplifying move, the effect on $\mathcal{F}$ is that three vertices get crushed together into one. This certainly will not increase the minimal vertex-graph-length of any homotopy class of loops in $T$, since any loop in the old vertex graph projects to one of the same or smaller length in the new vertex graph.

    Once all such moves have been exhausted, the meridians are still length 2, and $\mathcal{F}$ has all vertices of valence four. However, it still may not be standard if there are pairs of adjacent saddles with the same orientation.

    Since $\mathcal{F}$ has all vertices of valence four and has a repeated arc (the arcs of the length 2 meridian), we know due to the beginning of Section \ref{sec:repeatedArcs} that $\mathcal{F}$ has the form shown in Figure \ref{fig:fullFoliation}. Two saddles that are adjacent vertically will always have opposite orientations since they are connected to the same four vertices in opposite cyclic orders. Therefore if two adjacent saddles have the same orientation, they must appear side by side horizontally. Then the pair of saddles adjacent to these two vertically will also share an orientation. The sequence of simplifying moves depicted in Figure \ref{fig:contractRectangle} is available, which has the effect of crushing two sets of three vertices together.

    \begin{figure}
        \centering
        \includegraphics[width=0.9\linewidth]{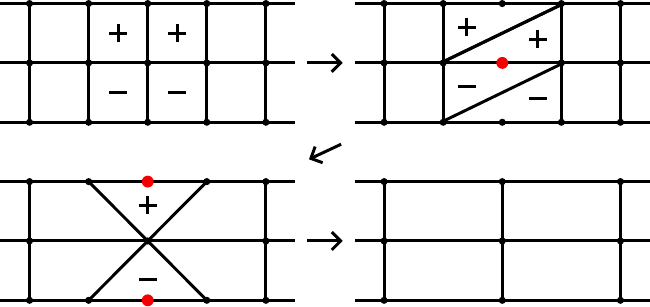}
        \caption{If two adjacent saddles have the same orientation in a ``height 2" foliation, then we can apply moves as above. For the second and third moves, a vertex of valence two is highlighted in red, which we simplify.}
        \label{fig:contractRectangle}
    \end{figure}

    This alteration clearly maintains the property of this component having a length 2 meridian. Since it decreases the binding weight of $T$, we may perform such simplifications until no more are available, at which point $\mathcal{F}$ is standard and every component of $T$ still has a length 2 meridian.
\end{proof}

So after subjecting $T$ to as many of our simplifying moves as possible, it has the properties necessary for Proposition \ref{prop:toNarrowPosition} to apply. In particular we can isotope $(L,T)$ such that $T$ is in narrow position at no further cost to the arc index of $L$.

\section{Initial binding weight}
\label{sec:bindingWeight}

We've seen in Sections \ref{sec:repeatedArcs} and \ref{sec:noRepeatedArcs} that we are able to achieve narrow position after a sequence of simplifying steps during which we raise the arc index by at most $\frac{1}{4}\alpha w_\beta + \alpha$. We desire a bound that depends only on $\alpha$, so we need an upper bound for the binding weight $w_\beta$ of $T$ in terms of $\alpha$.

To do this, we will show that if $T$ is a collection of JSJ satellite tori in the exterior of a link $L$ of arc index $\alpha$ (under some additional assumptions like $L$ being non-split and none of the tori being unknotted), we can position it initially with at most $2\alpha^4$ points of intersection with $S^1_\phi$, so that we only need to increase the arc index by $\frac{1}{2}\alpha^5 + \alpha$.

It is worth noting that the JSJ assumption is vital here. For any fixed limit on the binding weight, only finitely many isotopy classes of admissible satellite tori in the link exterior can be realised within that limit, whereas some knots and links have infinitely many isotopy classes of satellite tori in their exterior. Thus there is no hope of finding a bound that holds for any collection of satellite tori. By limiting ourselves to JSJ tori we avoid this issue.

To prove the claim, we will need the notions of a normal surface and a PL-admissible surface, as well as that of a generalised parallelity bundle.

Throughout this section, we fix a particular triangulation $\mathcal{T}$ of $S^3$ by $\alpha^2$ tetrahedra. Regarding $S^3$ as the join of two circles $S_\phi^1$ (our binding circle) and $S_\theta^1$, we can triangulate each of these circles by $\alpha$ 1-simplices, and take the join of these two triangulations to form $\mathcal{T}$.

This $\mathcal{T}$ is a triangulation in which $S_\phi^1$ and $S_\theta^1$ are both simplicial. Each 1-simplex that is part of one of these circles is incident to $\alpha$ tetrahedra, while all other 1-simplices are incident to only four tetrahedra.

The arc index of $L$ being $\alpha$ means that $L$ can be made simplicial in $\mathcal{T}$ as a union of $2\alpha$ 1-simplices, each of which travels from one circle to the other. Then $T$ is a surface embedded in $S^3$ that is disjoint from those 1-simplices that are part of $L$.

\subsection{Normal surfaces}

\begin{definition}
    An \emph{elementary normal disc} in a tetrahedron $t$ is a disc $D$ properly embedded in $t$ such that
    \begin{itemize}
        \item $\partial D$ is disjoint from the vertices of $t$ and meets the edges of $t$ transversely.
        \item $\partial D$ meets each edge of $t$ at most once, and meets each face of $t$ in at most one arc embedded in that face.
    \end{itemize}
\end{definition}

A tetrahedron supports only seven types of elementary normal disc: four types of triangle, and three types of squares (see Figure \ref{fig:normalDiscs}).

\begin{figure}
    \centering
    \includegraphics[width=0.5\linewidth]{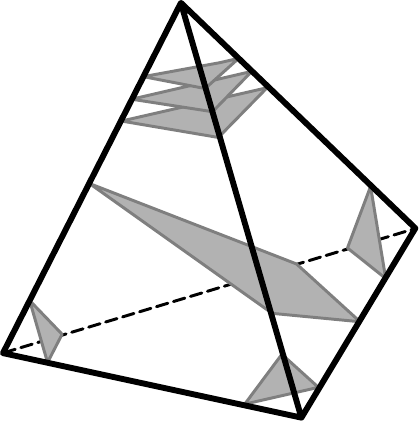}
    \caption{An example of a tetrahedron containing several normal discs. This tetrahedron contains every type of triangle, including three triangles of the same type, and one square. There are two other types of square, depending on which pair of opposite edges the square avoids, but a tetrahedron can only contain one type of square without normal discs intersecting with one another.}
    \label{fig:normalDiscs}
\end{figure}

\begin{definition}
    An embedded surface $S$ in a manifold $M$ is \emph{normal} with respect to a triangulation $\mathcal{T}$ of $M$ if the intersection of $S$ with each tetrahedron of $\mathcal{T}$ is a disjoint collection of normal discs.
\end{definition}

It is well-known that an incompressible surface in a closed irreducible 3-manifold can be isotoped to be normal with respect to any triangulation.

We will need a slight modification of this result, since our surface $T$ is not incompressible in our closed 3-manifold $S^3$.

\begin{lemma}
    \label{lem:normalise}
    Let $\mathcal{T}$ be a triangulation of a closed 3-manifold $M$, and $A \subseteq \mathcal{T}^{(1)}$ be a subset of the 1-skeleton which is a union of 1-simplices such that $M \setminus A$ is irreducible. Let $S$ be a closed embedded surface in $M$ that is disjoint from $A$ such that any compressing disc for $S$ intersects $A$, and no component of $S$ lies in a ball in $M \setminus A$. Then $S$ can be isotoped in $M \setminus A$ to be normal with respect to $\mathcal{T}$.
\end{lemma}
\begin{proof}
    Take an embedding of $S$ in $M \setminus A$ in this isotopy class with minimal weight (the number of points at which $S$ meets the 1-skeleton of $\mathcal{T}$), and subject to this restriction, the minimum possible number of components of intersection between $S$ and a tetrahedron that are not elementary normal discs.
    
    We claim that $S$ is normal. To show this, we explain how, if $S$ is not normal, then we can either reduce its weight or eliminate a non-normal-disc without increasing its weight. Since no component of $S$ lies in a ball in $M \setminus A$, any component of intersection of $S$ and a tetrahedron will have boundary.

    Suppose $S$ is not normal. Then there is a tetrahedron $\Delta$ where some component of $S \cap \Delta$ is either not a disc or is a disc but is not an elementary normal disc (meaning it meets no edges or meets some edge more than once).

    In the first case, we take an innermost component $\gamma$ of $S \cap \partial \Delta$ among those which are boundary components of non-discs of $S \cap \Delta$. Then $\gamma$ bounds a disc $D \subseteq \Delta$ with $D \cap S = \gamma$. This $D$ is disjoint from $A$ since $S$ is, so $D$ is not a compressing disc for $S$ and $\gamma$ must bound a disc $D' \subseteq S$. Now $D \cup D'$ is a sphere, so by irreducibility of $M \setminus A$ it bounds a ball disjoint from $A$. This ball's interior is disjoint from $S$, since otherwise some component of $S$ lies inside it, but we assumed that no component of $S$ lies in a ball in $M \setminus A$. Thus we can isotope $S$ over this ball, which has the effect of replacing $D'$ with $D$. If the non-disc belonged to $D'$, we have reduced the number of non-discs. If it did not, then we can push $D$ into $\Delta$ slightly. If $\gamma$ meets the 1-skeleton of $\mathcal{T}$, the weight of $S$ is reduced. Otherwise, the piece of $S$ that was on the other side of $\gamma$ was not a normal disc and has been eliminated, reducing the number of components of intersection between $S$ and a tetrahedron that are not elementary normal discs.

    In the case where some component of $S \cap \Delta$ is a disc meeting no edges, we take such a disc $D'$ with innermost boundary $\gamma$ in this face. Then $\gamma$ bounds a disc $D \subseteq \partial \Delta$ with $D \cap S=\gamma$. We can isotope $D'$ to $D$ and push it out of $\Delta$ slightly, reducing the number of discs that are not elementary normal discs (no component of $S$ lies between $D$ and $D'$ since no component lies in a ball in $M \setminus A$).

    Now assume that first two cases do not occur. In the final case, where some disc of $S \cap \Delta$ meets the same edge $e$ more than once, we take an innermost segment $\beta$ of $e$ that connects two points of the same circle $\gamma$ of $S \cap \partial \Delta$. Then there is $\alpha \subseteq \gamma$ so that $\beta \cup \alpha$ is a circle in $\partial \Delta$, which bounds a disc $D \subseteq \partial \Delta$. We push $D$ into $\Delta$ slightly, and since the other problematic cases do not occur, we may isotope it off of any other components of $S \cap \Delta$ that it meets. Let $D'$, $\alpha'$ be the new positions of $D$ and $\alpha$ after this isotopy. We have $D' \cap \partial \Delta = \beta$ and $D' \cap S = \alpha'$. By isotoping $S$ over a ball neighbourhood of $D'$, we may decrease the weight of $S$ by two (the two points $\alpha \cap \beta$ of $S \cap \mathcal{T}^{(1)}$ are eliminated).

    Therefore $S$ must be a normal surface, since we have seen that if it is not, this contradicts our minimality assumptions.
\end{proof}

By using Lemma \ref{lem:normalise} with $A$ = $L$ and $S$ = $T$, we will be able to put our collection of tori $T$ into normal form when necessary.

\subsection{PL-admissible surfaces}

The notion of a PL-admissible surface is defined in Section 6 of \cite{unknot}, and we emulate this approach here, although less details will be necessary since we need only consider the case of a closed surface. The idea is that by having a surface be in something close to admissible form, while also being normal and piecewise linear, we can use normal surface theory to gain control over its foliation.

First, since each tetrahedron of $\mathcal{T}$ is a join of two 1-cells, we can give each of these 1-cells the structure of a Euclidean straight line, so that each tetrahedron is Euclidean, and all of the gluing maps between tetrahedra are isometries.

Consider a closed normal surface $S$ with respect to our triangulation $\mathcal{T}$ of $S^3$. We can make this surface piecewise linear as follows.

First, we can isotope the surface so that each arc of intersection between $S$ and a 2-simplex is a straight line, without moving their endpoints. Next, each normal triangle can be taken to be flat (since its boundary is three straight arcs forming a loop in Euclidean space, the boundary is contained in a flat plane, and our triangle will also be contained in this plane). Normal squares will be realised as a pair of flat triangles meeting along a straight line between two opposite vertices of
the square. If we have this line run between the same two edges of the tetrahedron for all of the normal discs it contains, these squares will be disjoint. In the case where a normal square meets both $S_\phi^1$ and $S_\theta^1$, we will have this dividing line run from $S_\phi^1$ to $S_\theta^1$.

The result of this procedure is a surface in normal form with every normal triangle being flat, and every normal square being a pair of flat triangles. This surface has a singular foliation given by its intersection with each page $D_t$ . An example of some piecewise-linear normal discs with their foliations is shown in Figure \ref{fig:PLAdmissibleDiscs}.

\begin{figure}
    \centering
    \includegraphics[width=0.7\linewidth]{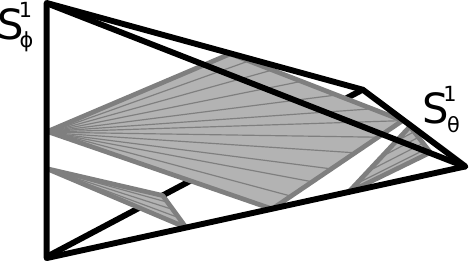}
    \caption{Some piecewise linear normal discs with their foliations. The closure of each page $D_t$ that meets the tetrahedron meets it in a flat triangle containing all of the edge in $S_\phi^1$, and meeting the edge in $S_\theta^1$ at a single point.}
    \label{fig:PLAdmissibleDiscs}
\end{figure}

After this procedure is carried out, we say that $S$ is in \emph{PL-admissible form}. This is not the original definition given in \cite{unknot}, but the surface we obtain will be PL-admissible in that sense.

The kinds of singularities that can arise in this foliation are similar to those that arise in the foliations of truly admissible surfaces. If a point of $S$ has a neighbourhood which is a disc foliated by arcs, we call it nonsingular. The points where $S$ intersects $S_\phi^1$ are called \emph{vertices}, and have leaves emerging from them in all directions. Points which have a disc neighbourhood where no other points have the same value of $\theta$ are called \emph{poles}. Points which have a disc neighbourhood in which the set of points having the same value of $\theta$ is a star-shaped graph are called \emph{generalised saddles}.

Essentially the singularities that can arise are the same as in the admissible case with the exception that saddles may have more than four separatrices emerging from them (though it will always be an even number). The layout of the foliation is related nicely to the normal structure of $S$ as follows.

\begin{lemma}
\label{lem:separatricesInSquares}
    For a closed surface $S$ in PL-admissible form, vertices of the foliation occur at precisely the points of $S \cap S_\phi^1$, poles and generalised saddles of the foliation can only occur at points of $S \cap S_\theta^1$, and every separatrix of the foliation is a diagonal of a single normal square, running from $S_\theta^1$ to $S_\phi^1$.
\end{lemma}
\begin{proof}
    Since $S$ meets $S_\phi^1$ transversely, vertices of the foliation are precisely points of $S \cap S_\phi^1$. Points in the interior of any normal disc are nonsingular, so any singularities must occur at points on the 2-skeleton of $\mathcal{T}$. One can see that if a point of $S$ lies in the 2-skeleton but not the 1-skeleton, it belongs to two normal discs, and will be nonsingular. If it lies on the 1-skeleton, but not on $S_\phi^1$ or $S_\theta^1$, then it is surrounded by four normal discs. Near to this point, the leaf containing it is given by the line segment along which two of these discs meet together with the line segment along which the other two discs meet, and the point is non-singular. Thus generalised saddles and poles can occur only at points of $S \cap S_\theta^1$.

    Now consider a separatrix of the foliation. This is a leaf that is incident to a generalised saddle. This generalised saddle is a point $s$ of $S \cap S_\theta^1$, and one can see that the only leaves ``emerging" from $s$ are diagonals of normal squares, and the other end of such a diagonal is a point of $S \cap S_\phi^1$, hence a vertex of the foliation. In other words, the saddle $s$ is contained in $\alpha$ different normal discs. Each of these discs that are normal triangles have no points other than $s$ in the same page as $s$, while each of these discs that are
    squares have many points in the same page as $s$, and these points form a leaf directly from $s$ to a vertex, so they form a separatrix.
\end{proof}

\pagebreak
In fact in our setting things are even simpler.

\begin{lemma}
\label{lem:noPoles}
    If $L$ is a link which is simplicial in $\mathcal{T}$ in an arc presentation with arc index $\alpha$ (so that it visits every 0-simplex exactly once and consists of $2\alpha$ 1-simplices each of which runs from one circle to the other), and $S$ is a surface disjoint from $L$ in PL-admissible form, then $S$ has no poles.
\end{lemma}
\begin{proof}
    First observe that in a general PL-admissible surface in $\mathcal{T}$, a pole $p$ must lie on $S \cap S_\theta^1$ and it can be incident to no normal squares (since then it would not be alone in its leaf). Thus it is surrounded in all $\alpha$ tetrahedra by normal triangles, and since these triangles glue together along their edges, they all separate a common 0-simplex $v$ from the rest of the vertices of the relevant tetrahedron. See Figure \ref{fig:pole} for a diagram.

    In our situation, this 0-simplex $v$, like all 0-simplices of $\mathcal{T}$, is a point of the link $L$, and exactly two of the 1-simplices leading from $v$ to $S_\phi^1$ are contained in $L$. As a result, some normal discs surrounding $p$ are not disjoint from $L$, contradicting our assumption.
    
    Thus $S$ can have no poles.
\end{proof}

\begin{figure}
    \centering
    \includegraphics[width=0.7\linewidth]{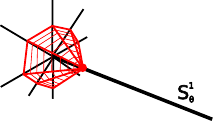}
    \caption{When a surface is in PL-admissible form with respect to $\mathcal{T}$, poles take the form of a point on $S_\theta^1$ incident in all $\alpha$ surrounding tetrahedra to a triangular normal disc.}
    \label{fig:pole}
\end{figure}

We saw in Section \ref{sec:admissible} that if the foliation of an admissible surface contains a vertex with valence less than 4, an isotopy can be performed reducing its binding weight. The same is true of PL-admissible surfaces. The following is Proposition 6.10 in \cite{unknot}. There it is stated for so-called ``deep" vertices, but since our PL-admissible surface is closed, all vertices are ``deep". Also there it is stated for not-necessarily-closed surfaces properly embedded in a polyhedral decomposition of the link exterior, while our surfaces are closed and embedded in $\mathcal{T}$ (but can also be thought of as being PL-admissible in this polyhedral decomposition of $S^3 \setminus N(L)$). Finally it gives bounds on the number of moves required to perform these reductions, which we will not need. Thus the following is a greatly weakened (but sufficient for our purposes) version of that Proposition.

\pagebreak
\begin{proposition}
\label{prop:makeStandard}
    Let $D$ be an arc presentation of a link $L$. Let $S$ be a PL-admissible surface with respect to $\mathcal{T}$.
    \begin{enumerate}
        \item If $S$ contains a 2-valent vertex, then there is a generalised exchange move on $D$, followed by an ambient isotopy of the link complement, taking $S$ to a surface $S'$ with
        $w_\beta(S') \leq w_\beta(S)-2$.
        \item If $S$ contains a 3-valent vertex, then there is a sequence of cyclic permutations, exchange moves, a generalised exchange move, and ambient isotopies of the link complement taking $S$ to a surface $S'$ with $w_\beta(S') \leq w_\beta(S)-2$.
    \end{enumerate}
\end{proposition}

\begin{corollary}
\label{cor:standard}
    If $S$ is a PL-admissible surface in the complement of an arc presentation $D$ of a link $L$ such that $S$ has minimal binding weight among all embeddings in its isotopy class and all arc presentations of $L$ with the same arc index as $D$, then $S$ cannot have any vertices of valence 2 or 3.
\end{corollary}

\subsection{Generalised parallelity bundles}

The final piece of normal surface theory we require will be the theory of generalised parallelity bundles. This will allow us to find a family of curves that fill each component of our surface while travelling across a controlled number of normal discs. A more thorough exposition of the topic can be found in Section 5 of \cite{compositeKnots} or Section 6 of \cite{triangulationFibred}. There generalised parallelity bundles are defined for normal surfaces in handle structures rather than triangulations, but the theory still works in our setting.

The idea is that when there are several normal discs of the same type in a tetrahedron, the space between two adjacent such discs forms a \emph{parallelity region}. By taking the union of all such regions in a particular component of the complement of the surface, we obtain a \emph{parallelity bundle} which is an $I$-bundle over some compact space $F$. The \emph{horizontal boundary} is the $\partial I$ bundle, and is a union of normal discs. The base space $F$ may not exactly be a surface, for instance if two parallelity regions meet at a corner and none of the other surrounding regions are parallelity regions. Even if $F$ is not a surface, a small neighbourhood of the bundle in the complement of the normal surface will be an $I$-bundle over a surface. We will frequently think of the bundle as being over this surface, although technically it may not be until we thicken it. Similarly the horizontal boundary may not be a subsurface of the normal surface, but we think of it as being one anyway, by taking a neighbourhood of it in the normal surface if necessary.

The reason this is helpful is because the number of normal discs of a component of the surface that do not appear in the horizontal boundary of the parallelity bundle in an adjacent component of the complement is bounded linearly in terms of the number of tetrahedra in the triangulation. This is because each tetrahedron can contain at most ten normal discs that aren't normally parallel to another normal disc in that tetrahedron on both sides.

Thus, if we can argue that curves we care about can be taken to lie outside of this horizontal boundary, we can obtain bounds on the number of normal discs they pass through, and therefore the number of separatrices they cross.

In order for us to establish the existence of important curves lying outside the horizontal boundary, we will need to be more permissive about what we consider to be a parallelity bundle.

\pagebreak
\begin{definition}
    Let $\mathcal{T}$ be our triangulation of $S^3$ in which $L$ is simplicial, and let $T$ be a collection of satellite tori in normal form. Then a \emph{generalised parallelity bundle} $\mathcal{B}$ is a subspace of one of the components of $S^3$ cut along $T$ such that
    \begin{enumerate}
        \item $\mathcal{B}$ (or a neighbourhood of it) is an $I$-bundle over a compact surface $F$.
        \item the $\partial I$-bundle is the intersection of this $I$-bundle with $T$.
        \item $\mathcal{B}$ is a union of components of tetrahedra cut along the normal discs of $T$.
        \item any such piece that the $I$-bundle over $\partial F$ runs alongside is a parallelity region between two parallel normal discs.
        \item $\mathcal{B}$ does not intersect $L$.
    \end{enumerate}    
\end{definition}

Roughly speaking, a generalised parallelity bundle looks like a parallelity bundle at the edges, but there may be parts of the interior that are not made of parallelity regions but nevertheless form an $I$-bundle.

This notion is useful because, as proved in \cite{compositeKnots}, as long as the surface $T$ does not admit a certain kind of simplifying move called an \emph{annular simplification}, the behaviour of a maximal generalised parallelity bundle is quite nice.

The definition of what exactly is considered an annular simplification is complicated and can be found in definition 5.4 of \cite{compositeKnots}, or in Section 6.4 of \cite{triangulationFibred}. What is important for our situation is that when an annular simplification is available, performing it involves replacing some annular subsurface $A$ of $T$ with a vertical boundary component of a generalised parallelity bundle, which will be another annulus with the same boundary as $A$. The surface obtained after the simplification is ambient isotopic to the surface before the simplification. If an annular simplification is available, performing it will strictly decrease the weight of the surface (this is Lemma 6.15 in \cite{triangulationFibred}). The idea is that the annulus we have replaced is now a union of vertical boundary components of parallelity regions between normal discs, so we can isotope it near the 1-skeleton of $\mathcal{T}$ to eliminate intersections of $T$ with the 1-skeleton.

In this way, annular simplifications strictly decrease the weight $|T \cap \mathcal{T}^{(1)}|$ of the surface, while not increasing its binding weight, and so by taking a least-weight representative of a surface we can assume none are available.

Proposition 5.6 and Corollary 5.7 in \cite{compositeKnots} gives us control over the bundles when no annular simplifications are available.

\begin{proposition}
\label{prop:bundles}
    Let $L$ and $T$ be an arc presentation of a non-split link and a normal collection of satellite tori as above, and suppose that $T$ admits no annular simplifications. Then if $M$ is one of the components of $S^3 \setminus T$, there is a generalised parallelity bundle $\mathcal{B}$ in $M$ such that $\mathcal{B}$ contains every parallelity region in $M$, and the horizontal boundary of $\mathcal{B}$ is incompressible in $M \setminus L$.
\end{proposition}

The idea behind this proof is that if the horizontal boundary of $\mathcal{B}$ were to have a compressing disc, then this disc could be used to either find an available annular simplification (which we prevent by assumption) or to ``plug a hole" in $\mathcal{B}$. This second situation can be prevented by taking a maximal generalised parallelity bundle containing all the parallelity regions in $M$.

Non-splitness of $L$ is necessary so that each $M$ is irreducible, which is an important assumption for the relevant results in \cite{compositeKnots}.

The horizontal boundary is homeomorphic to the orientable double cover of the base surface $F$ of the bundle. Thus if a component of the horizontal boundary was a torus, then some component of $F$ must be a torus or Klein bottle. There cannot be a Klein bottle component since $F$ is embedded in $S^3$. If $F$ were to have a torus component, then since $T \times I$ has no vertical boundary it would occupy all of $M$. But $M$ cannot be $T \times I$ since our tori cannot be parallel to each other (they bound disjoint solid tori and are not unknotted). Thus the horizontal boundary has no tori and instead is a collection of discs and non-trivial annuli in $T$.

Any vertical boundary component of a non-disc region of the bundle will be essential in the relevant piece of the link exterior cut along $T$, since a compressing disc would give a compressing disc of an annulus in the horizontal boundary, and if it was boundary-parallel, an annular simplification would be available.

\subsection{Binding weight bound}

With all of these tools, we are ready to bound the binding weight.

\begin{proposition}
\label{prop:bindingBound}
    Let $L$ be a satellite link with arc index $\alpha$, and $T$ be the collection of JSJ tori in its exterior. Then we can isotope $(L,T)$ so that $L$ is an arc presentation with arc index $\alpha$, and $T$ is PL-admissible with all vertices and saddles having valence four and with binding weight $w_\beta(T) \leq 2\alpha^4$.
\end{proposition}
\begin{proof}
    Take $(L,T)$ in their isotopy class that first of all has $L$ simplicial in $\mathcal{T}$ in an arc presentation of arc index $\alpha$, which minimises the binding weight $w_\beta(T)$ subject to this constraint, and which minimises the weight $w(T)$ (the number of intersections between $T$ and the 1-skeleton $\mathcal{T}^{(1)}$) subject to the first two constraints.

    Due to Lemma \ref{lem:normalise}, we can isotope $T$ in $S^3 \setminus L$ to make it normal, and then we can straighten it to be PL-admissible as we described above, and we may now speak of the foliation $\mathcal{F}$.
    
    From Lemma \ref{lem:noPoles}, we know $\mathcal{F}$ has no poles. We know from Corollary \ref{cor:standard} and our minimality assumptions that $\mathcal{F}$ has no vertices of valence less than four. We know from Lemma \ref{lem:separatricesInSquares} that no separatrix runs from a generalised saddle to itself. Then the separatrices form a graph in $\mathcal{F}$ whose complementary regions are all squares. By a similar Euler characteristic argument to the one we used in Section \ref{sec:admissibleTori}, since there are no vertices or saddles incident to less than four separatrices, there must be none incident to more than four separatrices. Thus in the foliation $\mathcal{F}$, all vertices and saddles have valence four.

    We refer to the components of $T$ as $T_1$, $T_2$, ..., $T_m$, and denote by $V_i$ the solid torus bounded by $T_i$. 

    By our minimality assumptions, no annular simplifications are available (since they decrease $w(T)$ without increasing $w_\beta(T)$). So from Proposition \ref{prop:bundles}, we have a generalised parallelity bundle $\mathcal{B}_i$ in each solid torus $V_i$, and another $\mathcal{B}_0$ in the final component $S^3 \setminus \bigcup_i V_i$, such that each contains every parallelity region in its component and has incompressible horizontal boundary.
    
    Each tetrahedron contains some number of normal squares, at most two of which are not parallel on both sides to other normal squares. Thus there are at most $2 \alpha^2$ such squares in total. Let $s_i$ be the number of such squares that are part of $T_i$. Then we have $\sum_i s_i \leq 2\alpha^2$.
    
    \textbf{Claim:} For each $i$, there is a pair of simple closed curves $\gamma_1,\gamma_2$ on $T_i$ that fill $T_i$ (meaning $T_i \setminus (\gamma_1 \cup \gamma_2)$ is a collection of discs) such that each $\gamma_j$ is disjoint from the singularities of $\mathcal{F}$ and crosses at most $s_i$ separatrices.

    \textbf{Proof of claim:} Let $H_1$, $H_2$ be the incompressible subsurfaces of $T_i$ that are the intersection of the horizontal boundaries of $\mathcal{B}_i$ and $\mathcal{B}_0$ with $T_i$.
    
    If one of these (call it $H$) is a collection of discs, then $G = T_i \setminus H$ is a torus with some holes. In particular, $G$ contains a simple closed curve which is non-trivial in $T_i$. We call this curve $\gamma_1$, and we can assume that it is transverse to the decomposition of $G$ into normal discs. We may also assume that it passes through each normal disc of $G$ at most once by the following argument: if it visits some disc more than once, we can ``splice" the curve in this disc to form two curves, and one of these must be non-trivial in $T_i$, so we can replace $\gamma_1$ with that one. This reduces the number of revisits, and so by repeating this process we find a non-trivial $\gamma_1$ with no revisits at all.

    Since $\gamma_1$ is non-trivial in $T_i$, it does not separate $G$, and so we can take another curve $\gamma_2$ that is transverse to the decomposition of $G$ into normal discs and meets $\gamma_1$ exactly once. By a similar argument to the above we can assume it visits each normal disc of $G$ at most once.

    The normal squares of $G$ are not adjacent in their tetrahedron on both sides to other normal squares, because if they were they would lie in $H$. Thus $G$ contains at most $s_i$ normal squares. Then each $\gamma_j$ visits at most $s_i$ normal squares, and since separatrices of $\mathcal{F}$ are diagonals of normal squares (Lemma \ref{lem:separatricesInSquares}), both $\gamma_j$ can be isotoped to cross at most $s_i$ separatrices. Also they fill $T$ since they meet exactly once.

    If neither $H_j$ is a collection of discs, then they both contain at least one non-trivial annulus. All of the annuli in $H_j$ are disjoint, so they must have the same slope. Consider first the case where the annuli-slopes for $H_1$ and $H_2$ are distinct. In this case, $G_j = T \setminus H_j$ is a collection of annuli with holes. Any of these will contain a curve that is non-trivial in $T_i$, which will have a slope agreeing with the annuli of $H_j$. As above we may take this curve to visit each normal disc of $G_j$ at most once, and call the curve $\gamma_j$.

    Then as above each $\gamma_j$ crosses at most $s_i$ separatrices of $\mathcal{F}$. Since their slopes differ (by assumption), $\gamma_1$ and $\gamma_2$ fill $T_i$.

    Finally, consider the case where both $H_j$ contain annuli of the same slopes. Then the vertical boundary components of $\mathcal{B}_i$ and $\mathcal{B}_0$ that are incident to the annular components of $H_j$ are essential annuli on both sides of $T_i$ with the same boundary slope. Since $T_i$ is a JSJ torus, the presence of these annuli indicates that the components of the JSJ decomposition of $S^3 \setminus L$ incident to $T_i$ on both sides are Seifert fibred (since hyperbolic pieces do not contain essential annuli). The annuli may be made vertical (a union of fibres) due to Proposition 1.12 of \cite{hatcher3M} (note that the only Seifert fibred spaces in which an annulus could be horizontal cannot appear in a JSJ decomposition, since JSJ tori are incompressible and non-parallel). The fact that the boundary slopes of these annuli agree means that the fibres of these Seifert fibred pieces agree on $T_i$, but then $T_i$ is not a JSJ torus, contradicting our assumption. The claim is proved. $\square$

    So we have the pair $\gamma_1$, $\gamma_2$ as described above. Since $\mathcal{F}$ has all vertices and saddles of valence four, the complement of the separatrices in $T_i$ is a collection of square tiles with four around each vertex or saddle. Thus we can put a Euclidean metric on $T_i$ such that each of these tiles is a right-angled unit square. Since $\gamma_1$, $\gamma_2$ cross at most $s_i$ separatrices, we can isotope them to be transverse, rectangular (a sequence of straight lines with right-angled corners) and have length at most $s_i$ with respect to this metric, as depicted in Figure \ref{fig:rectangularCurves}.

    \begin{figure}
        \centering
        \includegraphics[width=0.6\linewidth]{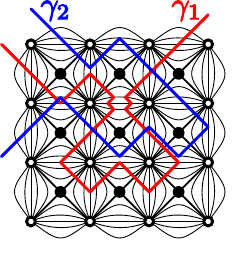}
        \caption{One of many ways in which our two curves can be made transverse and rectangular.}
        \label{fig:rectangularCurves}
    \end{figure}

    Since the curves fill $T$, the complement $T\setminus(\gamma_1 \cup \gamma_2)$ is a collection of discs with total perimeter $2l(\gamma_1)+2l(\gamma_2) \leq 4s_i$. By the Euclidean isoperimetric inequality, a disc's area $A$ and perimeter $L$ satisfy $4\pi A \leq L^2$, but we can get a slightly sharper bound if the disc's boundary is a rectangular curve.
    
    Consider a disc in $\mathbb{R}^2$ whose boundary is a rectangular curve (a sequence of horizontal and vertical line segments). The difference between the smallest and greatest vertical coordinates of the boundary is at most half of the total length of all the vertical arcs, which we call $v$. The difference between the smallest and greatest horizontal coordinates of the boundary is at most half of the total length of all the horizontal arcs, which we call $h$. Then the disc is contained in a $v/2$ by $h/2$ rectangle, so its area is at most $vh/4$, and $v+h = L$. Maximising $vh$ subject to $v+h = L$, we find that $16A \leq L^2$.

    Using this improved (16 is a slightly better constant than $4\pi$) isoperimetric inequality, we find that $T_i$ is a union of discs $D_j$ with areas $A_j$ satisfying $16A_j \leq L_j^2$ and $\sum L_j \leq 4 s_i$. The area of $T_i$ is $\sum A_j$, and the maximum possible value this can take is $s_i^2$.

    The area of $T_i$ with respect to this Euclidean metric is the number of square tiles in the foliation of $T_i$, which is exactly twice the number of vertices of $\mathcal{F}$ on $T_i$. Thus, the binding weight of $T_i$ is at most $s_i^2/2$.

    Then since we know that $\sum s_i \leq 2\alpha^2$, the binding weight $w_\beta$ of $T$ is at most $\sum s_i^2/2 \leq (2\alpha^2)^2/2 = 2\alpha^4$.
\end{proof}

We are now ready to prove Theorem \begin{NoHyper}\ref{thm:satelliteDiagram}\end{NoHyper}.

\pagebreak
\begin{repeatThmSatelliteDiagram}
    If $L$ is a non-split link of arc index $\alpha$ contained in a disjoint union of solid tori $V = \bigsqcup V_i$ such that every $T_i = \partial V_i$ is JSJ in the exterior of $L$ and no $V_i$ is unknotted, then there is an isotopy of $(L,T)$ in $S^3$ such that afterwards $L$ is in an arc presentation of arc index at most $\frac{1}{2}\alpha^5+2\alpha$ and the corresponding grid diagram is a satellite diagram of a grid diagram corresponding to a collection of core curves of each $V_i$.
\end{repeatThmSatelliteDiagram}
\begin{proof}
    Proposition \ref{prop:bindingBound} tells us that we can position $L$ and $T$ such that $L$ is an arc presentation with arc index $\alpha$ and $T$ is PL-admissible with binding weight at most $2\alpha^4$.

    We perturb $T$ slightly to make it smooth and admissible, and then we use the techniques we have described throughout the paper to simplify $T$ until it is in narrow position. First we proceed as described in Section \ref{sec:noRepeatedArcs}, introducing a length 2 meridian to each $T_i$. Then we use the simplifying moves described in Section \ref{subsec:simplifyingMoves} to standardise the foliation. Finally we proceed as described in Section \ref{sec:repeatedArcs} and $T$ will now be in narrow position. None of these techniques increase the arc index of $L$ except for the one described in Section \ref{sec:noRepeatedArcs}, which as we note in Section \ref{subsec:arcIndexIncrease} will increase the arc index by at most $\frac{1}{4}\alpha w_\beta + \alpha \leq \frac{1}{2}\alpha^5+\alpha$. Together with the $\alpha$ arcs we begin with, we see that once $T$ has been simplified into narrow position, the arc index of $L$ will be at most $\frac{1}{2}\alpha^5 + 2\alpha$.

    As mentioned in Section \ref{subsec:narrowPosition}, once $T$ is in narrow position, by just pushing some arcs of $L$ through time we obtain an arc presentation of $L$ whose associated grid diagram $D$ is a satellite diagram following the grid diagram corresponding to an arc presentation of the core curves of $V$.
\end{proof}

\section{Bounding the complexity of the satellite}
\label{sec:bound}

Theorem \ref{thm:satelliteDiagram} gives us an opportunity to establish a relationship between the complexity of a satellite link and the complexities of the companion and the pattern, as well as the wrapping number.

The theorem guarantees the existence of a satellite diagram of our satellite link with an upper bound on complexity. In this section we put a lower bound on its complexity by examining the structure of a satellite diagram.

The structure of a satellite diagram is shown in Figure \ref{fig:satelliteDiagram}.

\begin{figure}
    \centering
    \includegraphics[width=0.9\linewidth]{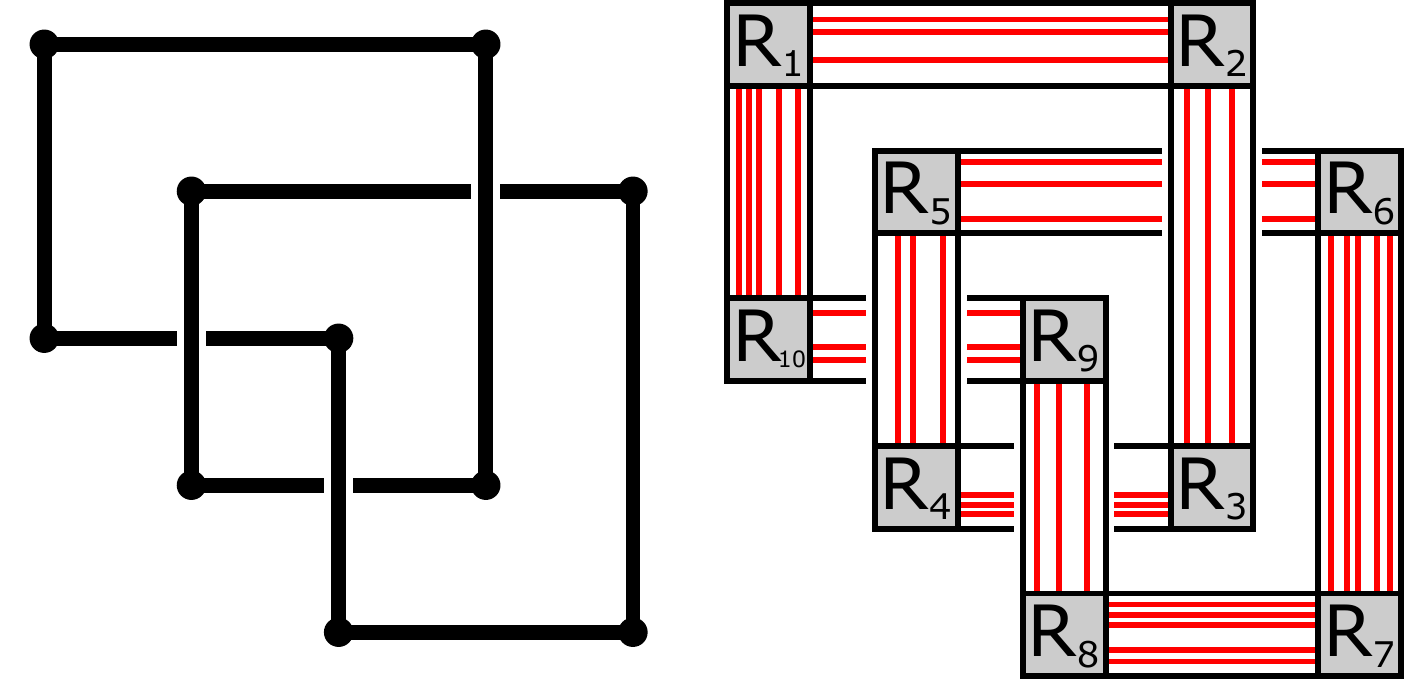}
    \caption{On the left is a grid diagram $D$ of the trefoil knot. On the right is the form of a satellite diagram following it. A satellite diagram consists of arbitrary configurations within the shaded neighbourhoods $R_i$ of the corners of $D$, and collections of parallel arcs following the neighbourhoods of the arcs of $D$.}
    \label{fig:satelliteDiagram}
\end{figure}

We now put a lower bound on the crossing number of this diagram. As can be seen from Figure \ref{fig:satelliteDiagram}, crossings will be of two kinds. Some crossings appear inside the shaded squares, while others appear outside of these squares in groups where the crossings of the companion diagram were. We count the second type first.

\begin{definition}
    The \emph{wrapping number} $w$ of a satellite knot is a property of the pattern $P$, and is defined to be the minimal number of points of intersection between $P$ and any meridian disc of the solid torus. A satellite link in a union of several solid tori may have a different wrapping number in each one, and we denote by $w$ the minimum of these.
\end{definition}

The satellite diagram $D'$ has a parallel collection of arcs following each arc of the companion diagram $D$. Each arc of $D$ is followed by at least $w$ parallel arcs of $D'$.

This is because if some arc of $D$ had fewer than $w$ arcs of $D'$ following it, we would be able to find a meridian disc in the corresponding part of the solid torus that is pierced by $L$ fewer than $w$ times. This is impossible by definition of $w$.

Therefore, each crossing of $D$ gives a region of $D'$ in which at least $w$ parallel vertical arcs cross over at least $w$ parallel horizontal arcs. So there are at least $w^2c(D)$ crossings of this kind, which is at least $w^2c(C)$, where $C$ is the companion link.

Now consider the crossings of the other kind. These occur inside the neighbourhoods of the corners of $D$. ``Deforming" and ``unwinding" the diagram as illustrated in Figure \ref{fig:unwindPattern} yields a link diagram of the pattern living in a disjoint union of annuli (that represent the solid tori), with some framings, in which all the crossings come from the shaded squares. Thus the number of such crossings is at least $c(P)$, by which we mean the minimal crossing number of any diagram of the pattern $P$ that can be drawn in a union of annuli representing the solid tori, with any framing.

\begin{figure}
    \centering
    \includegraphics[width=1\linewidth]{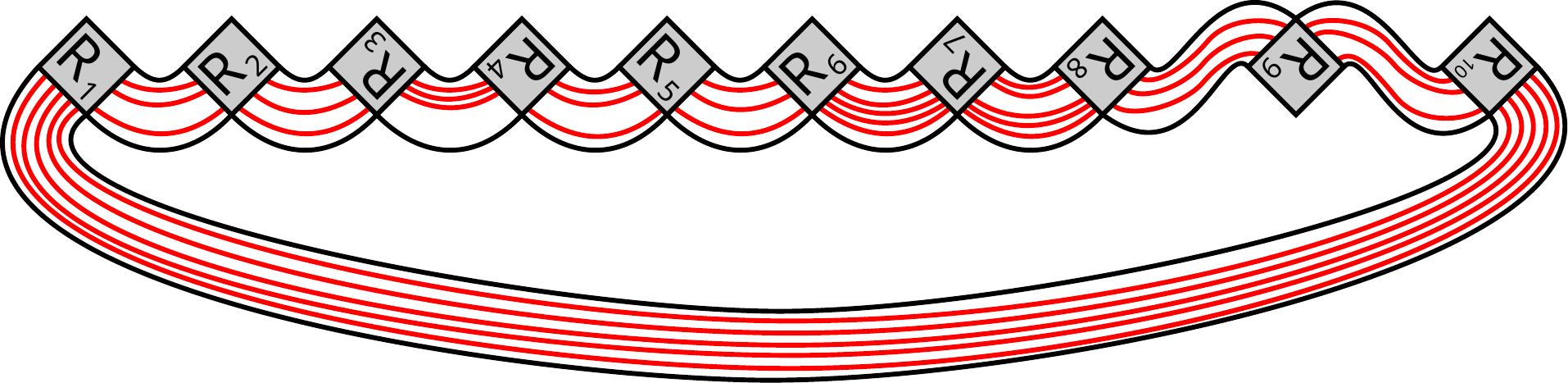}
    \caption{The satellite diagram from Figure \ref{fig:satelliteDiagram}, unwound to produce a knot diagram for the pattern knot $P$, drawn in an annulus representing the solid torus it lies in. The regions $R_i$ are unchanged, though they are moved and rotated, and all crossings in this diagram occur inside of these regions. If the companion had more than one component, this unwinding would yield several annular diagrams which we would place side by side.}
    \label{fig:unwindPattern}
\end{figure}

So altogether, a satellite grid diagram with companion link $C$ and pattern link $P$ contains at least $w^2c(C)+c(P)$ crossings.

In addition to crossing number bounds, we can also get arc index bounds. Observe that $D'$ has arc index at least $w\alpha(C)$, since each horizontal arc of $D$ becomes a set of at least $w$ horizontal arcs in $D'$, and $D$ has arc index at least $\alpha(C)$ since it is a grid diagram of $C$. It is also the case that $D'$ has arc index at least $\alpha(P)$, and this can be seen as follows:

Any grid diagram can be converted to the standard grid diagram of an unlink if we allow ourselves to perform not just destabilisation and the usual exchange moves but also ``illegal" row and column exchange moves in which we swap two rows or columns even if they interleave. If we perform such an ``illegal" move on $D$ and perform the corresponding operation on $D'$, the effect is depicted in Figure \ref{fig:illegalExchange}. It is as if we homotope $V$ in $S^3$, allowing it to pass through itself. This will change the link type of $D'$, but importantly it will not change the way $L$ sits inside $V$ up to homeomorphism (the framing may change), nor will it increase the arc index of $D'$.

\begin{figure}
    \centering
    \includegraphics[width=1.0\linewidth]{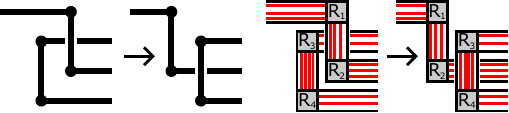}
    \caption{This ``illegal column exchange" alteration to $D$ (on the left) and $D'$ (on the right) changes the link type but does not change the way $L$ sits in $V$.}
    \label{fig:illegalExchange}
\end{figure}

If we use a sequence of these unrestricted moves to reduce $D$ to the standard unlink diagram with the appropriate number of components, and perform the corresponding moves on $D'$, we will obtain a diagram of $P$ as a satellite of the standard unlink diagram with the same number of arcs as $D'$. Thus $D'$ has at least $\alpha(P)$ arcs. 

Some final results we need to get our bounds relate arc index and crossing number. In Section 5 of \cite{Cromwell_Nutt_1996}, it is argued that a grid diagram with arc index $\alpha$ can have at most $\alpha^2/2$ crossings; the uppermost horizontal arc can cross no vertical ones, the next can cross two (since only two vertical arcs reach that high), and so on increasing until the central arc(s), then decreasing again in the lower half. The actual bound obtained is slightly tighter, and is a different expression depending on the parity of $\alpha$, but we will use $\alpha^2/2$ for simplicity. It is also argued that by redirecting horizontal arcs ``the other way around" if doing so would reduce the number of crossings, one can find that for any link $L$, $c(L) \leq \alpha(L)^2/4$.

Also, in Section 4 of \cite{Cromwell_Nutt_1996} is the theorem that if a link diagram $D$ satisfies certain conditions, it can be used to build a grid diagram of the same link with arc index at most $c(D)+2$. In \cite{monotonicSimplification}, at the start of the proof of Theorem 2, it is mentioned that the same method shows that if $D$ is any connected diagram, there is a grid diagram for the same link with arc index at most $2c(D)+2$ (the idea of the proof is to take a knot diagram and draw a binding circle onto it). This shows that for any non-split link $L$, $\alpha(L) \leq 2c(L)+2$.

Now we are ready to prove Theorem \begin{NoHyper}\ref{thm:JSJbound}\end{NoHyper}.

\begin{repeatThmJSJBound}
    If $L$ is a non-split link contained in a disjoint union of solid tori $V = \bigsqcup V_i$ such that every $T_i = \partial V_i$ is JSJ in the exterior of $L$ and no $V_i$ is unknotted, and $C$, $P$, and $w$ are the corresponding companion link, pattern link, and minimal wrapping number around any $V_i$, then
    \[
        w^2c(C)+c(P) \leq \frac{1}{7}(2c(L)+2)^{10}
    \]
    and
    \[
        \max(w\alpha(C),\alpha(P)) \leq \frac{1}{2}\alpha(L)^5 + 2 \alpha(L)
    \]
    where by $c(P)$, $\alpha(P)$ we mean the minimal crossing number or arc index of a (grid) diagram of $P$ drawn in a collection of disjoint annuli (we place no restriction on the framing with which $P$ is drawn).
\end{repeatThmJSJBound}
\begin{proof}
    From Theorem \ref{thm:satelliteDiagram}, we know that $(L,T)$ can be isotoped into narrow position, giving a satellite diagram $D'$ with arc index at most $\frac{1}{2}\alpha(L)^5+2\alpha(L)$. By the result from \cite{Cromwell_Nutt_1996} and the fact that $\alpha(L) \leq 2c(L)+2$, we find that

    \[
        c(D') \leq \frac{1}{2}\left(\frac{1}{2}\alpha(L)^5+2\alpha(L)\right)^2
        \leq \frac{1}{2}\left(\frac{1}{2}(2c(L)+2)^5+2(2c(L)+2)\right)^2
    \]
    which is at most $\frac{1}{7}(2c(L)+2)^{10}$ for $c(L) \geq 1$ (and grows asymptotically like $\frac{1}{8}(2c(L)+2)^{10}$).

    However, we know by the discussion earlier that $D'$ has at least $w\alpha(C)$ arcs, at least $\alpha(P)$ arcs, and at least $w^2c(C)+c(P)$ crossings, so we obtain

    \[
        w^2c(C)+c(P) \leq \frac{1}{7}(2c(L)+2)^{10}
    \]
    and
    \[
        \max(w\alpha(C),\alpha(P)) \leq \frac{1}{2}\alpha(L)^5 + 2 \alpha(L).
    \]
\end{proof}

\section{Extending to non-JSJ satellite tori}
\label{sec:nonJSJ}

In previous sections we have established a lower bound on the complexity of a satellite in terms of the complexity of the companion link, the wrapping number, and the complexity of the pattern, but only in the case where these arise from a collection of satellite tori which are all JSJ in the link exterior.

In this section we show that specifically in the satellite knot case (when $L$ has only one component, and $T$ is a single torus), we can gain some control over what happens when $T$ is not a JSJ torus.

To do this we need results about the structure of the JSJ decomposition of knot complements. Such results can be found in Budney's survey \cite{budney2006}. In particular, Theorem 4.18 from this paper gives us the following:
    
\begin{proposition}
\label{prop:SFStypes}
    In the JSJ decomposition of a knot exterior, the Seifert fibred pieces are of the following forms:
    \begin{itemize}
        \item The base orbifold is a disc with two exceptional fibres.
        \item The base orbifold is an annulus with one exceptional fibre.
        \item The space is a circle bundle over a disc with holes.
    \end{itemize}
\end{proposition}

This result is not true for link exteriors, which is why we must specialise to the knot case.

We will also make use of Lemma 1 of \cite{budney2006}, which states

\begin{lemma}
\label{lem:solidTorusSomewhere}
    If $M$ is a connected compact irreducible submanifold of $S^3$ with non-empty boundary a union of tori, then either $M$ is a solid torus or some component of $\overline{S^3 \setminus M}$ is a solid torus.
\end{lemma}

\begin{proposition}
\label{lem:nonJSJConnSum}
    Let $T$ be a non-JSJ satellite torus for a satellite knot $K$. Then the associated companion knot $C$ (a longitude of $T$) is a connected sum of knots that appear as longitudes of distinct JSJ tori, and these JSJ tori are separated from $K$ by $T$.
\end{proposition}
\begin{proof}
    The torus $T$ can be isotoped to be disjoint from the JSJ tori, at which point it is contained in one of the Seifert fibred pieces (since the hyperbolic pieces contain no essential tori) and is not boundary-parallel in that piece (since it is not JSJ), nor is it compressible. If the base space of the Seifert fibred piece is a disc with two singular points or an annulus with one singular point, then any incompressible torus is boundary-parallel. So due to Proposition \ref{prop:SFStypes}, $T$ must be in a piece $M$ which is a circle bundle over a disc with holes (at least three holes, since otherwise all incompressible tori are boundary-parallel).

    \textbf{Claim:} Exactly one component of $\overline{S^3 \setminus M}$ is a solid torus, and the fibres of $M$ in the component of $\partial M$ incident to this solid torus component are meridian curves.

    \textbf{Proof of claim:} Since $M$ itself is not a solid torus, Lemma \ref{lem:solidTorusSomewhere} tells us that at least one of the components of $\overline{S^3 \setminus M}$ is a solid torus. To see that only one is, we observe that this solid torus has a JSJ torus as its boundary. If the solid torus does not contain $K$, then this JSJ torus is compressible in $S^3 \setminus K$, which is a contradiction. Thus any solid torus component of $\overline{S^3 \setminus M}$ must contain a component of $K$. Since $K$ is a knot, it has only one component, so there is exactly one solid torus component $V$ of $\overline{S^3 \setminus M}$.

    Suppose for a contradiction that the fibres of $M$ on $\partial V \subseteq \partial M$ are not meridian curves. Then we can extend the Seifert fibration of $M$ to a Seifert fibration of $M' = M \cup V$. Then $M'$ is a Seifert fibred space with boundary, so is irreducible (indeed there are only three reducible Seifert fibred spaces, and none of them have boundary, as shown in Proposition 1.13 of \cite{hatcher3M}). Also $M'$ is not a solid torus (it still has at least three boundary components), so Lemma \ref{lem:solidTorusSomewhere} tells us that some component of $\overline{S^3 \setminus M'}$ is a solid torus. But the components of $\overline{S^3 \setminus M'}$ are exactly the components of $\overline{S^3 \setminus M}$ excluding the only one which is a solid torus. This is a contradiction, so in fact the fibres of $M$ are meridian curves of $\partial V$, and the claim is proved. $\square$

    We denote by $S$ the component of $\partial M$ that is incident to the solid torus. When we imagine the disc with holes that $M$ is fibred over, we view $S$ as being the outermost boundary component. We can isotope $T$ to be vertical (a union of fibres) due to Proposition 1.12 of \cite{hatcher3M}, so it is a union of fibres over a circle. Then $T$ as well as every boundary component of $M$ has the property that fibres are meridian curves of the solid torus it bounds in $S^3$, since a meridian disc can be formed by taking the fibres over an arc connecting the relevant circle to $S$ together with a meridian disc of the solid torus on the other side of $S$. Thus a longitude for each torus is given by a horizontal circle lying in a single disc-with-holes slice of $M$.

    If one draws a diagram of the disc with holes together with the circle over which $T$ lies, as on the left of Figure \ref{fig:nonJSJ}, then after an ambient isotopy of this diagram in the plane we obtain a diagram as on the right of Figure \ref{fig:nonJSJ}. From this it can be seen that a longitude of $T$ is a connected sum of the longitudes of those components of $\partial M$ that $T$ separates from $S$ (and therefore from $K$, which lies beyond $S$). The relevant connected sum spheres can be obtained by taking the fibres over an arc meeting $C$ twice with both endpoints on $S$ together with two meridian discs of the solid torus $S$ bounds.

    \begin{figure}
        \centering
        \includegraphics[width=0.9\linewidth]{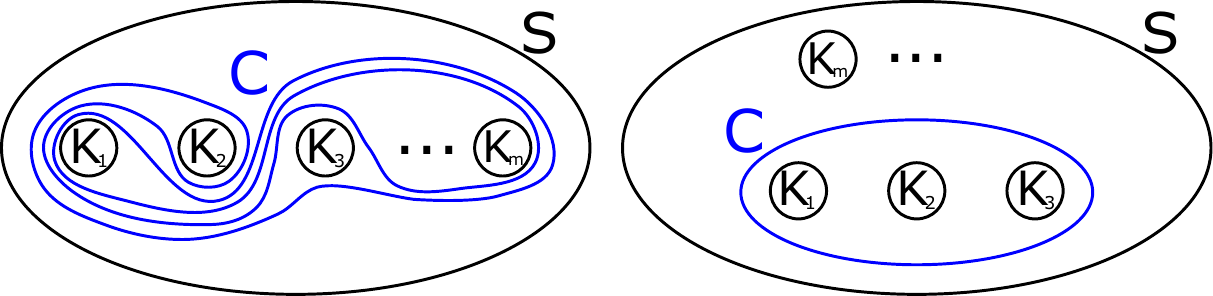}
        \caption{An example of performing an ambient isotopy of a diagram of the disc with holes. In this case we see that the knot $C$ appearing as a horizontal circle of the blue torus $T$ is a connected sum of $K_1$,$K_2$, and $K_3$, which are the longitudes of the tori that $T$ separates from $S$.}
        \label{fig:nonJSJ}
    \end{figure}

    Thus $C$ is a connected sum of knots that appear as longitudes of distinct JSJ tori which are separated from $K$ by $T$.
\end{proof}

Thus we can bound the crossing number of a non-JSJ companion knot in terms of the crossing numbers of the JSJ companion knots, as well as the number of JSJ tori in the exterior of the satellite knot.

\begin{lemma}
\label{lem:howManyJSJ}
    If $K$ is a knot, then the number of JSJ tori in its exterior is at most $24c(K)-1$.
\end{lemma}
\begin{proof}
    First, form a partially ideal triangulation of the knot exterior using $4c(K)$ tetrahedra. To see that this is possible, one can observe that in Section 3 of \cite{compositeKnots}, Lackenby uses a knot diagram $D$ to construct a handle structure for the knot exterior, having $4c(D)$ 0-handles, which is dual to a partially ideal triangulation of the knot exterior by $4c(D)$ tetrahedra. Alternatively, in Section 3 of \cite{Weeks2005}, Weeks describes how to build an ideal triangulation for a knot exterior using $4c(K)+4$ tetrahedra, and the final four tetrahedra are unnecessary if only a partially ideal triangulation is required.
    
    Suppose there are $n$ JSJ tori. Since the JSJ tori, taken together as a surface with many components, is incompressible, it can be put into normal form with respect to this partially ideal triangulation. Since any torus in $S^3$ is separating, this splits $S^3 \setminus K$ into $n+1$ pieces. Each piece is a union of pieces of tetrahedra cut along the normal discs. Since no two JSJ tori are parallel, no JSJ piece can consist solely of parallelity regions. Thus each JSJ piece contains at least one non-parallelity piece of a tetrahedron. But each tetrahedron has at most six such pieces, so $n+1$ is at most six times the number of tetrahedra.

    Thus $n < 24c(K)-1$.
\end{proof}

We can control the wrapping number associated with the relevant JSJ tori in terms of our non-JSJ torus.

\begin{lemma}
\label{lem:wrappingInequality}
    Let $K$ be a satellite knot, and $T_1$, $T_2$ be disjoint satellite tori in the exterior of $K$ such that $T_1$ and $K$ lie on opposite sides of $T_2$. Let $w_1$, $w_2$ be the wrapping numbers corresponding to $T_1$ and $T_2$ respectively. Then $w_1 \geq w_2$.
\end{lemma}
\begin{proof}
    Let $D$ be a meridian disc of $T_1$ in $S^3$ with the number of components of $D \cap T_2$ minimal. Since $T_1$ is a satellite torus, $D \cap K$ is non-empty, and then since $T_2$ separates $T_1$ from $K$, the collection of circles $D \cap T_2$ is non-empty. Let $\gamma$ be an innermost such circle in $D$, which bounds a disc $D' \subseteq D$ with $D' \cap T_2 = \gamma$. 
    
    If $D'$ is not a compressing disc for $T_2$, then $\gamma$ bounds a disc $D'' \subseteq T_2$. Then $D' \cup D''$ bounds a ball on both sides, and we can isotope $D'$ to $D''$ across one of these balls, then push $D''$ off $T_2$ to reduce the number of components of $D \cap T_2$.

    Since we assumed the number of components was minimal, $D'$ must be a compressing disc (a meridian disc) for $T_2$. Thus $K$ meets $D'$ in at least $w_2$ points. Since $D' \subseteq D$, $K$ meets $D$ in at least $w_2$ points. Therefore $w_1 \geq w_2$.
\end{proof}

Finally we are ready to prove Theorem \begin{NoHyper}\ref{thm:bound}\end{NoHyper}.

\begin{repeatThmBound}
    If $K$ is a satellite knot with companion knot $C$ and wrapping number $w$, then
    \[
        c(K) \geq \frac{2}{5} (w^2c(C))^{1/11}
    \]
    and
    \[
        \alpha(K) \geq \frac{4}{5} (w \alpha(C))^{1/7}.
    \]
\end{repeatThmBound}
\begin{proof}
    Let $T$ be the satellite torus such that the corresponding companion knot and wrapping number are $C$ and $w$. In the case where $T$ is a JSJ torus, the desired inequality follows from Theorem \ref{thm:JSJbound}, in which a stronger inequality was proved.

    If $T$ is not a JSJ torus, then from Lemma \ref{lem:nonJSJConnSum} we know that there are distinct JSJ tori $T_1$, $T_2$, ..., $T_n$ that are separated from $K$ by $T$, with associated companion knots and wrapping numbers $C_1$, $C_2$, ..., $C_n$ and $w_1$, $w_2$, ..., $w_n$ such that $C = C_1 \sharp C_2 \sharp \ldots \sharp C_n$.

    Since the $T_i$ are distinct, Lemma \ref{lem:howManyJSJ} tells us that $n < 24c(k)-1$. Since each $T_i$ is separated from $K$ by $T$, Lemma \ref{lem:wrappingInequality} tells us that $w_i \geq w$ for each $i$. So we have

    \[
        w^2c(C) \leq \sum_{i=1}^n w^2c(C_i) \leq \sum_{i=1}^n w_i^2c(C_i).
    \]

    Since each $T_i$ is JSJ, Theorem \ref{thm:JSJbound} gives us 

    \[
        w_i^2c(C_i) \leq \frac{1}{7} (2c(K)+2)^{10}.
    \]

    Putting this all together,

    \[
        w^2c(C) \leq \sum_{i=1}^n \frac{1}{7} (2c(K)+2)^{10} \leq \frac{24c(K)-1}{7}(2c(K)+2)^{10}.
    \]

    Since every satellite knot has $c(K) \geq 6$, $2c(K) + 2 \leq \frac{7}{3}c(K)$, and we have $w^2c(C) \leq \frac{24}{7} (\frac{7}{3})^{10} c(K)^{11} \leq 16402 c(K)^{11}$. Thus we obtain our bound
    \[
        c(K) \geq \left(\frac{w^2 c(C)}{16402}\right)^{1/11} \geq \frac{2}{5} (w^2 c(C))^{1/11}.
    \]

    Similarly

    \[
        w \alpha(C) \leq \sum_{i=1}^n w\alpha(C_i) \leq \sum_{i=1}^n w_i \alpha(C_i).
    \]

    Theorem \ref{thm:JSJbound} gives us $w_ic(C_i) \leq \frac{1}{2} \alpha(K)^5+2\alpha(K)$, and from \cite{Cromwell_Nutt_1996} we know that $c(K) \leq \alpha(K)^2/4$, so

    \[
        w\alpha(C) \leq (24c(K)-1)\left(\frac{1}{2}\alpha(K)^5 + 2 \alpha(K)\right)
    \] \[
        \leq 6 \alpha(K)^2 \left(\frac{1}{2}\alpha(K)^5 + 2\alpha(K)\right) \leq 3\alpha(K)^7 + 12\alpha(K)^3 \leq 4\alpha(K)^7
    \]
    and so we have the other bound

    \[
        \alpha(K) \geq \left(\frac{w\alpha(K)}{4}\right)^{1/7} \geq \frac{4}{5} (w\alpha(K))^{1/7}.
    \]
\end{proof}

\section{Supplementary Results}
\label{sec:supplementals}

In this final section we make two interesting observations that follow quickly from our discussion so far.

\begin{theorem}
    If $L$ is a non-split link contained in a disjoint union of solid tori $V = \bigsqcup V_i$ such that every $T_i = \partial V_i$ is JSJ in the exterior of $L$ and no $V_i$ is unknotted, and $C$ is the corresponding companion link, then

    \[
        \alpha(C) \leq 2\alpha(L)^4
    \]
    and
    \[
        c(C) \leq (2c(L)+2)^8.
    \]
\end{theorem}
\begin{proof}
    From Proposition \ref{prop:bindingBound} we may isotope $(L,T)$ into a position where $w_\beta(T) \leq 2\alpha(L)^4$. Due to lemma \ref{lem:twoPageMeridian}, we know that each component of $T$ has a meridian curve that lies in a pair of pages. In particular these meridian curves, thought of as loops in the vertex graph of $\mathcal{F}$, visit each vertex at most once. Thus on each component we can find a longitude (a loop meeting the meridian once algebraically) in the vertex graph that visits each vertex at most once.

    Then the union of all these longitudes has the link type of $C$ and is an arc presentation which meets the binding circle in at most $2\alpha(L)^4$ points. So $\alpha(C) \leq 2\alpha(L)^4$.

    The crossing number bound comes from this together with $c(C) \leq \alpha(C)^2/4$ and $\alpha(L) \leq 2c(L)+2$.
\end{proof}

Compared to the bound $c(C) \leq 10^{13}c(K)$ from \cite{satelliteBound}, and the arc index bound that can be obtained from it, the bounds we have just proved have a significantly worse growth rate. However, they may be of interest for the more reasonable coefficients, the applicability to some links in addition to knots, and the considerably different proof.

\begin{theorem}
    \label{thm:algorithm}
    Given a knot $K$, the following is an algorithm to decide whether $K$ is a satellite knot:
    \begin{enumerate}
        \item Choose a grid diagram $D$ of $K$. Call its arc index $\alpha$.
        \item Enumerate all grid diagrams that can be reached from $D$ by elementary moves without raising the arc index above $\alpha^5/2 + 2\alpha$ at any stage.
        \item Check whether each of these is a non-trivial satellite diagram (in which the companion is non-trivial, $K$ does not lie in a ball in the relevant solid torus, and is not a core curve of this solid torus). If one is, $K$ is a satellite. If none are, $K$ is not.
    \end{enumerate}
\end{theorem}
\begin{proof}
    If $K$ is a satellite knot, it will have some knotted JSJ torus in its exterior. Then the discussion throughout the paper gives a procedure for transforming a diagram $D$ of $K$ with arc index $\alpha$ into a satellite diagram of $K$ with arc index at most $\alpha^5/2 + 2\alpha$, and at no point during this process are any destabilisations necessary. The whole process can be realised as a sequence of elementary moves that never raise the arc index beyond $\alpha^5/2 + 2\alpha$. Thus if $K$ is satellite, this diagram will appear in the enumeration.

    All that remains to be seen is that it is indeed possible to decide whether a particular grid diagram is a non-trivial satellite diagram. We can try to draw boxes around the corners of the diagram in a suitable way; there are only finitely many configurations to try. If a configuration looks promising, we need only check that the companion is non-trivial, that $K$ does not lie in a ball inside the torus, and that $K$ is not a core curve for the torus. There are various ways to check these, and at least in the case of the first two checks, there are even methods based on grid diagrams, suiting the theme of the algorithm.

    We have a grid diagram of the companion. We may enumerate the finitely many grid diagrams that can be reached from it by grid moves excluding stabilisation. By \cite{monotonicSimplification}, the standard unknot diagram will be absent if and only if the companion is non-trivial.

    If we introduce to our candidate diagram of $K$ an extra component that follows a meridian of the relevant torus, this 2-component link is split if and only if $K$ is contained in a ball in the torus. We may enumerate the finitely many grid diagrams that can be reached from this one by grid moves excluding stabilisation. Again by \cite{monotonicSimplification}, the link is split if and only if there is a split diagram in this list.

    Finally, note that it is decidable whether $K$ is a core curve of the solid torus. For instance because the solid torus with a deleted open neighbourhood of $K$ can be triangulated, and by work of Haraway in \cite{TxIisDecidable} there is an algorithm to determine whether this space is homeomorphic to $T \times I$.
\end{proof}

It is already known that satellite detection is decidable, for instance because the JSJ decomposition of a manifold can be computed using the techniques of Jaco and Tollefson described in \cite{algorithmicDecomposition}. In fact more is known; in \cite{NPwithGRH}, Baldwin and Sivek show that satellite recognition lies in NP, assuming the generalised Riemann Hypothesis. Subsequently in \cite{NPwithoutGRH}, Haraway and Hoffman were able to prove the same result unconditionally.

A common theme among these algorithms is having some way to decide whether a given torus is a satellite torus or not (and into this issue Theorem \ref{thm:algorithm} offers no new insight). Once this is known to be decidable, all that is needed to describe an algorithm for satellite knot recognition is some finite and algorithmically enumerable collection of candidate tori which is guaranteed to contain a satellite torus if one exists at all. Given such a list, a satellite recognition algorithm can simply check each of these tori in turn.

One such collection is the list of ``fundamental" normal surfaces in a triangulation of the knot exterior, as noted in \cite{NPwithoutGRH}. The main novel observation of Theorem \ref{thm:algorithm} is that an alternative such collection is given by those tori that can be nicely drawn on one of a finite (and algorithmically enumerable) list of grid diagrams.

\printbibliography

\end{document}